\documentclass[11pt, twoside]{article}
\pdfoutput=1

\usepackage{graphicx}
\usepackage[caption=false]{subfig}
\usepackage{amsmath}
\usepackage{amssymb,amsfonts}
\usepackage{amsthm}
\usepackage{bm}
\usepackage{mathrsfs}
\usepackage{amssymb}
\usepackage{multirow}

\newcommand{\norm}[1]{\left\lVert#1\right\rVert}

\usepackage[margin=1in]{geometry}

\usepackage{algorithm}
\usepackage{algorithmic}

\numberwithin{equation}{section}
\usepackage{multirow}
\usepackage{makecell}
\usepackage{booktabs}
\theoremstyle{definition}
\newtheorem{theorem}{Theorem}

\newtheorem{lemma}{Lemma}

\newtheorem{remark}{Remark}

\usepackage{cite}
\usepackage{hyperref}
\usepackage[nameinlink]{cleveref}

\usepackage{fancyhdr}
\begin{document}
	
	\title{A New Ensemble HDG Method for Parameterized Convection Diffusion PDEs}

\author{Gang Chen%
	\thanks{School of Mathematics, Sichuan University, Chengdu, China (\mbox{cglwdm@scu.edu.cn}).}
	\and
	Liangya Pi%
	\thanks{Department of Mathematics and statistical, Missouri University of Science and Technology, Rolla, MO, USA (\mbox{lpp4f@mst.edu}).}
	\and
	Yangwen Zhang%
	\thanks{Department of Mathematics Science, University of Delaware, Newark, DE, USA (\mbox{ywzhangf@udel.edu}).}
}

\date{\today}

\maketitle

\begin{abstract}
	{We devised a first order time stepping ensemble hybridizable discontinuous Galerkin (HDG) method for a group of parameterized convection diffusion PDEs with  different  initial conditions, body forces, boundary conditions and coefficients in our earlier work \cite{PiChenZhangXu1}. We obtained an optimal convergence rate for the ensemble solutions in $L^\infty(0,T;L^2(\Omega))$ on a simplex mesh; and obtained a superconvergent rate for the ensemble solutions in $L^2(0,T;L^2(\Omega))$ after an element-by-element postprocessing if polynomials degree $k\ge 1$ and the coefficients of the PDEs are independent of time. In this work, we propose a new second order time stepping ensemble HDG method to revisit the problem.  We obtain a superconvergent rate for the ensemble solutions in $L^\infty(0,T;L^2(\Omega))$ without an element-by-element postprocessing for all polynomials degree $k\ge 0$. Furthermore, our mesh can be any polyhedron, no need to be simplex; and the coefficients of the PDEs can dependent on time. Finally, we present numerical experiments to  confirm our theoretical results.}
\end{abstract}

\section{Introduction}
In this work, we propose a new second order time stepping  ensemble hybridizable discontinuous Galerkin (HDG) method to efficiently simulate a group of parameterized convection diffusion equations on a Lipschitz polyhedral domain  $\Omega\subset \mathbb{R}^{d} $ $ (d\geq 2)$. For $j=1,2, \ldots, J$,  find $(\bm q_j, u_j)$ satisfying
\begin{equation}\label{convection_pde}
\begin{aligned}
c_j\bm{q}_j +\nabla u_j&= 0 \quad  \  \mbox{in} \; \Omega\times (0,T],\\
\partial_tu_j+{\nabla}\cdot \bm{q}_j
+\bm\beta_j\cdot\nabla u_j&= f_j \quad  \mbox{in} \; \Omega\times (0,T],\\
u_j&=g_j \quad  \mbox{on} \; \partial\Omega\times (0,T],\\
u_j(\cdot,0)&=u_j^0 \ \ \  \;\mbox{in}\ \Omega,
\end{aligned}           
\end{equation}
where 
$c_j:=c_j(\bm x,t)$,
$f_j:=f_j(\bm x,t)$, 
$g_j:=g_j(\bm x,t)$, 
$\bm \beta_j:=\bm\beta_j(\bm x,t)$,
and $u_j^0:=u_j^0(\bm x)$ are given functions.

For many computational applications in real life, one needs to solve a group of PDEs with different input conditions, like the applications in petroleum engineering, which need to predict the transport properties of rock core-sample in centimeter scale. We need to capture the flow capacity of every single nanopore with different inputs,  and the porous media of shale core-sample is composed of more than $10^6$ pores. However, to efficiently simulate a group of PDEs with different inputs is a great challenge. 

A first order time stepping ensemble method was proposed by \cite{Jiang_Layton_Flow_UQ_2014} to study a set of $J$ solutions of the Navier-Stokes equations with different initial conditions and forcing terms. The $J$ solutions are computed simultaneously by solving a linear system with one common coefficient matrix and multiple RHS vectors. This leads to a great computational efficiency in linear solvers when either the LU factorization (for small-scale systems) or a block iterative algorithm (for large-scale systems) is used. Later,  a second order time stepping ensemble algorithm was designed by \cite{Jiang_Fluid_JSC_2015}. Recently,  a new ensemble method was proposed to treat the PDEs have different coefficients, see \cite{Gunzburger_Jiang_Wang_Flow_IMAJNA_2018,Gunzburger_Jiang_Wang_Flow_CMAM_2017}. The ensemble method has been applied to many different models; see, e.g., \cite{Jiang_NS_NUPDE_2017,Jiang_Layton_Fluid_NMPDE_2015,Jiang_Tran_Flow_CMAM_2015,Gunzburger_Jiang_Schneier_NS_SINUM_2017,Gunzburger_Jiang_Schneier_NS_IJNAM_2018,Fiordilino_Bousinesq_SINUM_2018,Luo_Wang_Heat_SINUM_2018}. It is worthwhile to mention that the previous works only obtained a \emph{suboptimal} $L^\infty(0,T;L^2(\Omega))$ convergence rate  for the ensemble solutions. 

More recently, we proposed a first order time stepping ensemble hybridizable discontinuous Galerkin (HDG) method in \cite{PiChenZhangXu1} to study a group of convection diffusion PDEs with different initial conditions, boundary conditions, body forces and coefficients. We obtained an \emph{optimal} $L^\infty(0,T; L^2(\Omega))$ convergence rate  for the solutions on a simplex mesh,  and  we obtained a $L^2(0,T; L^2(\Omega))$ superconvergent rate  if  the polynomials of degree $k\ge 1$ and the coefficients of the PDEs are independent of time. This  ensemble HDG method uses polynomials of degree $k$ for all variables, i.e., the flux variables $\bm q_j $ and  the scalar variables $u_j$.

In this work, we devise a new second order time stepping ensemble HDG method, this  method uses  polynomials degree $k$ to approximate the fluxes  and the numerical traces, while using polynomials degree $k+1$ to approximate the scale variable. This method was proposed by      \cite{Lehrenfeld_PhD_thesis_2010} and  later analyzed by \cite{Oikawa_Poisson_JSC_2015} for a single  steady elliptic PDE, they obtained a  superconvergent rate for the scalar variable for all $k\ge 0$.  This HDG method has been  extended to study the PDEs with a convection term by \cite{Qiu_Shi_Convection_Diffusion_JSC_2016,Qiu_Shi_NS_IMAJNA_2016}.  However, the superconvergent rate was lost when $k=0$.

In this paper, we first restore the superconvergence for  $k=0$ by modifying the stabilization function in \cite{Qiu_Shi_Convection_Diffusion_JSC_2016}. Next, we show that the new ensemble HDG method can obtain a  $L^\infty(0,T;L^2(\Omega))$ superconvergent rate for all $k\ge 0$ on a general polyhedron mesh and  without assume the coefficients are independent of time. It is worth  mentioning  that  this new ensemble HDG method keep the advantages of  the ensemble methods, i.e., all realizations share one common coefficient matrix and multiple RHS vectors at each time step, which can be solved efficiently by some exist solvers as we mentioned previously.

The paper is organized as follows. We introduce the  improved HDG formulation and the ensemble HDG method in \Cref{EnsembleHDGFormulation}. Next, we give some preliminary materials  and  prove the ensemble HDG method is unconditionally stable in \Cref{Stability}.  Then we  give a rigorous error analysis in \Cref{errroanalysis}. Finally, we provide some numerical experiments to confirm our  theoretical result in \Cref{numericalexperiments}.

\section{The Ensemble HDG Formulation}
\label{EnsembleHDGFormulation}

The HDG methods were proposed by \cite{Cockburn_Gopalakrishnan_Lazarov_Unify_SINUM_2009}, which  are based on a mixed formulation and introduce a numerical flux and a numerical trace to approximate the flux and the trace of the solution. The global system involves the numerical trace only since we can element-by-element eliminate the numerical flux and the  solution. Therefore, the  HDG methods have a significantly smaller number of globally coupled degrees of freedom compare to DG methods.  The HDG methods have been extended to many models;  see, e.g., \cite{Chen_Cockburn_Convection_Diffusion_MathComp_2014,Cockburn_Shi_Stokes_MathComp_2013,Sanchez_Ciuca_Nguyen_Peraire_Cockburn_Hamiltonian_JCP_2017,Cockburn_Zhixing_Hungria_Waves_JSC_2018,Vidal_Nguyen_Peraire_electromagnetic_JCP_2018}. We emphasize that the HDG method in this work is considered to be a superconvergent method. Specifically, if polynomials of degree $k \ge 0$ are used for the numerical traces (global system), then we can obtain $k+2$ order for the scalar variables; see, e.g., \cite{Qiu_Shen_Shi_elasticity_MathComp_2018,Qiu_Shi_Convection_Diffusion_JSC_2016,Qiu_Shi_NS_IMAJNA_2016}. Hence, from the viewpoint of globally coupled degrees of freedom, this method achieves superconvergence for the scalar variable.

To describe the ensemble HDG method, we introduce some notation.  Let $\mathcal{T}_h$ be a collection of disjoint  shape regular polyhedral $K$ that partition $\Omega$. Here by shape regular we refer to \cite{Chen_Monk_Peter1}. Let $\partial \mathcal{T}_h$ denote the set $\{\partial K: K\in \mathcal{T}_h\}$. For an element $K$ of the collection  $\mathcal{T}_h$, let $e = \partial K \cap \partial \Omega$ denote the boundary face of $ K $ if the $d-1$ Lebesgue measure of $e$ is non-zero. For two elements $K^+$ and $K^-$ of the collection $\mathcal{T}_h$, let $e = \partial K^+ \cap \partial K^-$ denote the interior face between $K^+$ and $K^-$ if the $d-1$ Lebesgue measure of $e$ is non-zero. Let $\mathcal{E}_h^o$ and $\mathcal{E}_h^{\partial}$ denote the sets of interior and boundary faces, respectively, and let $\mathcal{E}_h$ denote the union of  $\mathcal{E}_h^o$ and $\mathcal{E}_h^{\partial}$.   For each $K\in\mathcal{T}_h$, let $h_K$  denote the diameter of the smallest $d$-dimensional ball contain $K$, and $h=\max_{K\in\mathcal{T}_h}h_K$.
We finally set
\begin{align*}
	(w,v)_{\mathcal{T}_h} := \sum_{K\in\mathcal{T}_h} (w,v)_K,   \quad\quad\quad\quad\langle \zeta,\rho\rangle_{\partial\mathcal{T}_h} := \sum_{K\in\mathcal{T}_h} \langle \zeta,\rho\rangle_{\partial K},
\end{align*}
where  $(\cdot,\cdot)_K$  and $\langle \cdot,\cdot\rangle_{\partial K}$ denote the standard $L^2$ inner product.

For any integer $k\ge 0$, let $\mathcal{P}^k(K)$ denote the set of polynomials of degree at most $k$ on the element $K$.  We  recall  the standard $L^2$  projection operators $\Pi_{\ell}: L^2(K)\to \mathcal P^{\ell}(K)$ and $P_M : L^2(e)\to \mathcal P^{k}(e)$ satisfying
\begin{subequations}\label{L2}
	\begin{align}
		(\Pi_{\ell} u, w)_K &= (u,w)_K & \forall w\in \mathcal P^{\ell}(K),\label{L2_do}\\
		\langle P_M  u, \mu \rangle_e &= \langle u,\mu \rangle_e  & \forall \mu \in \mathcal P^{k}(e).\label{L2_edge}
	\end{align}
\end{subequations} 
Moreover, the vector $L^2$ projection $\bm \Pi_{\ell}$ is defined similarly.

We consider the discontinuous finite element spaces:
\begin{align*}
	\bm{V}_h  &:= \{\bm{v}\in [L^2(\Omega)]^d: \bm{v}|_{K}\in [\mathcal{P}^k(K)]^d, \forall K\in \mathcal{T}_h\},\\
	{W}_h  &:= \{{w}\in L^2(\Omega): {w}|_{K}\in \mathcal{P}^{k+1}(K), \forall K\in \mathcal{T}_h\},\\
	{M}_h(g)  &:= \{{\mu}\in L^2(\mathcal{E}_h): {\mu}|_{e}\in \mathcal{P}^k(e), \forall e\in \mathcal{E}_h,\mu|_{\mathcal{E}_h^\partial} = P_M g\}.
\end{align*}

For $ w_h \in W_h $ and $ \bm r_h \in \bm V_h $, let $ \nabla v_h $ and $ \nabla \cdot \bm r_h $ denote the gradient of $ w_h $ and the divergence of $ \bm r_h $ applied piecewise on each element $K\in \mathcal T_h$.

\subsection{The Improved HDG Method} 
Next, we consider the spatial  semidiscretization for \eqref{convection_pde} by an improved  HDG method.  For all $j = 1,2,\ldots, J$,  find $(\bm q_{jh},u_{jh},\widehat u_{jh})\in \bm V_h\times W_h\times M_h(g_j)$ satisfying 
\begin{equation}\label{semi_discretization}
\begin{split}
(c_j\bm{q}_{jh},\bm r_j)_{\mathcal T_h} - (u_{jh}, \nabla\cdot\bm r_j)_{\mathcal T_h}  + \langle \widehat u_{jh}, \bm r_j\cdot\bm n\rangle_{\partial \mathcal T_h}&= 0,\\
(\partial_tu_{jh}, w_j)_{\mathcal T_h} - (\bm{q}_{jh},\nabla w_j)_{\mathcal T_h} + \langle \widehat {\bm q}_{jh}\cdot \bm n, w_j\rangle_{\partial\mathcal T_h} &\\
-( \bm \beta_j u_{jh},\nabla w_j)_{\mathcal T_h} - ((\nabla\cdot\bm \beta_j) u_{jh} , w_j)_{\mathcal T_h}+\langle \bm\beta_j\cdot\bm n \widehat u_{jh}, w_j\rangle_{\partial \mathcal T_h} &=   (f_j,w_j)_{\mathcal T_h} ,\\
\langle \widehat {\bm q}_{jh}\cdot \bm n,  \mu_j\rangle_{\partial \mathcal T_h}&=0
\end{split}
\end{equation}
for all $(\bm r_j, w_j,\mu_j)\in \bm V_h\times W_h\times M_h(0)$. The numerical traces on $\partial\mathcal{T}_h$ are defined as
\begin{equation}\label{stabilization_1}
\begin{split}
\widehat{\bm{q}}_{jh}\cdot \bm n &=\bm q_{jh}\cdot\bm n + h_K^{-1} (P_M u_{jh}-\widehat u_{jh}).
\end{split}
\end{equation}

\begin{remark}
	The stabilization functions  in \cite{Qiu_Shi_Convection_Diffusion_JSC_2016} are defined as following
	\begin{equation}\label{stabilization_2}
	\widehat{\bm{q}}_{jh}\cdot \bm n =\bm q_{jh}\cdot\bm n + h_K^{-1} (P_M u_{jh}-\widehat u_{jh})  + \tau_j^C (u_{jh}-\widehat u_{jh}),
	\end{equation}
	where $\tau_j^C$  are positive stabilization functions defined on $\partial \mathcal T_h$. Compare with our stabilization function \eqref{stabilization_1}, a upwind term in \eqref{stabilization_2} was added to guarantee  the wellpossdness but destroy the superconvergence when $k=0$; see, e.g., \cite{Qiu_Shi_Convection_Diffusion_JSC_2016} for a single convection diffusion PDE  and \cite{HuShenSinglerZhangZheng_HDG_Dirichlet_control3} for an optimal control problem.
\end{remark}

\subsection{The Ensemble HDG Formulation}
It is easy to see that the system \eqref{semi_discretization}-\eqref{stabilization_1} has $J$ different coefficient matrices since $c_j^n$ and $\bm{\beta}_j^n$ are different for each $j$, the superscript $n$ denotes the function value at the time $t_n$. The main idea of the ensemble algorithms is change the variables $c_j^n$ and  $\bm \beta_j^n$ into their ensemble means: 
\begin{align}\label{ensemblemena}
	\overline c^n = \frac 1J \sum_{j=1}^J c_j^n \qquad \textup{and} \qquad
	\overline{\bm{\beta}}^n = \frac 1J \sum_{j=1}^J \bm{\beta}_j^n.
\end{align}
Next, we suppose the time domain is uniformly partition into $N$ steps and the time step is $\Delta t:=T/N$. Let $t_n := n\Delta t$ for $n=2,3,\cdots, N$,  we define 
\begin{align*}
	\partial^+_t w^n = \frac{3w^n-4w^{n-1}+w^{n-2}}{2\Delta t}.
\end{align*}

For all $j=1,2,\ldots, J$ and $n=2,3,\cdots,N$,   our second order time stepping ensemble HDG method finds  $(\bm q^n_j,u^n_j,\widehat u^n_j)\in \bm V_h\times W_h\times M_h(g_j)$  satisfying
\begin{subequations}\label{full_discretization1}
	\begin{align}\label{full_discretizationa}
		(\overline c^n\bm{q}^{n}_{jh},\bm{r}_j)_{\mathcal{T}_h}-(u^n_{jh},\nabla\cdot \bm{r}_j)_{\mathcal{T}_h}+\langle\widehat{u}^n_{jh},\bm{r}_j\cdot \bm n \rangle_{\partial{\mathcal{T}_h}}=((\overline c^n-c_j^n)(2\bm{q}^{n-1}_{jh}-\bm{q}^{n-2}_{jh}),\bm{r}_j)_{\mathcal{T}_h},
	\end{align}
	for all $\bm r_j\in \bm V_h$, and 
	\begin{align}\label{full_discretizationb}
		\begin{split}
			&\qquad (\partial^+_tu^n_{jh},w_j)_{\mathcal T_h}-(\bm{q}^n_{jh}, \nabla w_j)_{\mathcal{T}_h} 	+\langle\widehat{\bm{q}}^n_{jh}\cdot \bm{n},w_j\rangle_{\partial{\mathcal{T}_h}} \\
			&\qquad -( \nabla\cdot\overline{\bm\beta}^n u_{jh}^n,w_j)_{\mathcal{T}_h}
			-(\overline{\bm{\beta}}^nu_{jh}^n,\nabla w_j)_{\mathcal{T}_h}+\langle (\overline{\bm{\beta}}^n\cdot\bm n) \widehat{u}^n_{jh},v_j \rangle_{\partial\mathcal{T}_h}\\
			&= \quad (f^n_j,w_j)_{\mathcal{T}_h}
			-([\nabla\cdot(\overline{\bm\beta}^n-\bm\beta_j^n)](2u_{jh}^{n-1}-u_{jh}^{n-2}),w_j)_{\mathcal{T}_h}\\
			&\qquad -((\overline{\bm{\beta}}^n-\bm\beta_j^n)(2u_{jh}^{n-1}-u_{jh}^{n-2}),\nabla w_j)_{\mathcal{T}_h}\\
			&\qquad +\langle [(\overline{\bm{\beta}}^n-\bm{\beta}_j^n)\cdot\bm n] ( 2\widehat{u}^{n-1}_{jh}-\widehat{u}^{n-2}_{jh}),w_j \rangle_{\partial\mathcal{T}_h}
		\end{split}
	\end{align}
	for all $w_j \in W_h$, and 
	\begin{align}\label{full_discretizationc}
		\langle \widehat{\bm q}^n_{jh}\cdot\bm n,\mu_j \rangle_{\partial\mathcal{T}_h}=0
	\end{align}
	for all $\mu_j\in M_h(0)$ and the numerical fluxes  are defined by
	\begin{align}\label{flux}
		\widehat{\bm q}_{jh}^n\cdot\bm n=\bm q_{jh}^n\cdot\bm n
		+h_K^{-1}(P_M u_{jh}^n-\widehat{u}_{jh}^n).
	\end{align}
\end{subequations}

To start up the  second order time stepping ensemble HDG system \eqref{full_discretization1}, besides the initial condition $(\bm q_{jh}^0,u_{jh}^0, \widehat u_{jh}^0)$, we need the information of $(\bm q_{jh}^1,u_{jh}^1,\widehat u_{jh}^1)$.  We take the initial conditions $u_{jh}^0 = \Pi_{k+1} u_0$, $ \bm q_{jh}^0=-\nabla u_{jh}^0/c_j^{0}$. Since  $u_{jh}^0$ is double-valued on $\mathcal E_h$, then restict $u_{jh}^0$ on $\mathcal E_h$ is doulbe valued. Therefore, we only take one as the initial condition for $\widehat u_{jh}^0$. Followed in \cite{Gunzburger_Jiang_Wang_Flow_CMAM_2017},  $(\bm q_{jh}^1,u_{jh}^1, \widehat u_{jh}^1)$ is computed by the following first order ensemble HDG method:
\begin{subequations}
	\begin{align*}
		(\overline c^1\bm{q}^1_{jh},\bm{r}_j)_{\mathcal{T}_h}-(u^1_{jh},\nabla\cdot \bm{r}_j)_{\mathcal{T}_h}+\langle\widehat{u}^1_{jh},\bm{r}_j\cdot \bm n \rangle_{\partial{\mathcal{T}_h}}=((\overline c^1-c_j^1)\bm{q}^0_{jh},\bm{r}_j)_{\mathcal{T}_h}
	\end{align*}
	for all $\bm r_j\in \bm V_h$, and 
	\begin{align*}
		\begin{split}
			&\qquad ((u^1_{jh}-u^0_{jh})/ \Delta t,w_j)_{\mathcal T_h}-(\bm{q}^1_{jh}, \nabla w_j)_{\mathcal{T}_h} 	+\langle\widehat{\bm{q}}^1_{jh}\cdot \bm{n},w_j\rangle_{\partial{\mathcal{T}_h}} \\
			&\qquad -( \nabla\cdot\overline{\bm\beta}^1 u_{jh}^1,w_j)_{\mathcal{T}_h}
			-(\overline{\bm{\beta}}^1u_{jh}^1,\nabla w_j)_{\mathcal{T}_h}+\langle (\overline{\bm{\beta}}^1\cdot\bm n) \widehat{u}^1_{jh},v_j \rangle_{\partial\mathcal{T}_h}\\
			& =\quad(f^1_j,w_j)_{\mathcal{T}_h}
			-([\nabla\cdot(\overline{\bm\beta}^1-\bm\beta_j^1)]u_{jh}^0,w_j)_{\mathcal{T}_h}\\
			&\qquad -((\overline{\bm{\beta}}^1-\bm\beta_j^1)u_{jh}^0,\nabla w_j)_{\mathcal{T}_h}
			+\langle [(\overline{\bm{\beta}}^1-\bm{\beta}_j^1)\cdot\bm n]  \widehat{u}_{jh}^0,w_j \rangle_{\partial\mathcal{T}_h}
		\end{split}
	\end{align*}
	for all $w_j\in W_h$. 
\end{subequations}

\begin{lemma}\label{hdg}
	System \eqref{full_discretizationa}-\eqref{flux} is equivalent  to the  following system
	\begin{subequations}\label{full_discretion_ensemble}
		\begin{align}
			(\overline c^n\bm{q}^{n}_{jh},\bm{r}_j)_{\mathcal{T}_h}-(u^n_{jh},\nabla\cdot \bm{r}_j)_{\mathcal{T}_h}+\langle\widehat{u}^n_{jh},\bm{r}_j\cdot \bm n \rangle_{\partial{\mathcal{T}_h}} 
			=
			((\overline c^n-c_j^n)(2\bm{q}^{n-1}_{jh}-\bm{q}^{n-2}_{jh}),\bm{r}_j)_{\mathcal{T}_h},\label{full_discretion_ensemble_a}
		\end{align}
		\begin{align}\label{full_discretion_ensemble_b}
			\begin{split}
				&\quad (\partial^+_tu^n_{jh},w_j)_{\mathcal T_h}+(\nabla\cdot\bm{q}^n_{jh},  w_j)_{\mathcal{T}_h}
				-\langle\bm q_{jh}^n\cdot\bm n, \mu_j\rangle_{\partial\mathcal{T}_h}-((\nabla\cdot\overline{\bm\beta}^n) u_{jh}^n,w_j)_{\mathcal{T}_h}\\
				&\quad 
				-(\overline{\bm{\beta}}^nu_{jh}^n,\nabla w_j)_{\mathcal{T}_h}+\langle (\overline{\bm{\beta}}^n\cdot\bm n) \widehat{u}^n_{jh},w_j \rangle_{\partial\mathcal{T}_h}+\langle h_K^{-1} (P_M u_{jh}^{n}-\widehat{u}_{jh}^n),P_M w_j-\mu_j\rangle_{\partial{\mathcal{T}_h}}\\
				&=(f^n_j,w_j)_{\mathcal{T}_h}-([\nabla\cdot(\overline{\bm\beta}^n-\bm\beta_j^n)](2u_{jh}^{n-1}-u_{jh}^{n-2}),w_j)_{\mathcal{T}_h}\\
				&\quad -((\overline{\bm{\beta}}^n-\bm\beta_j^n)(2u_{jh}^{n-1}-u_{jh}^{n-2}),\nabla w_j)_{\mathcal{T}_h}+\langle [(\overline{\bm{\beta}}^n-\bm{\beta}_j^n)\cdot\bm n] (2\widehat{u}^{n-1}_{jh}-\widehat{u}^{n-2}_{jh}),w_j \rangle_{\partial\mathcal{T}_h}
			\end{split}
		\end{align}
		for all $(\bm r_j, w_j,\mu_j)\in \bm V_h\times W_h\times M_h(0)$. 
	\end{subequations}
\end{lemma}
The proof of \Cref{hdg} is simply by substituting \eqref{flux} into \eqref{full_discretizationa}-\eqref{full_discretizationc}, substracting \eqref{full_discretizationc} from \eqref{full_discretizationb} and using integration by parts.

\section{Stability}
\label{Stability}

In this paper, we use the standard notation $W^{m,p}(D)$ for Sobolev spaces on $D$ with norm $\|\cdot\|_{m,p,D}$ and seminorm $|\cdot|_{m,p,D}$.  We use  $H^{m}(D)$ instead of $W^{m,p}(D)$ when $p=2$.  We omit the index $p$ and $D$ in the corresponding norms and seminorms when $p=2$ or $D = \Omega$.  Also, we omit the index $m$ when $m=0$ in the corresponding norms.  We denote by $C(0,T;W^{m,s}(\Omega))$ the Banach space of all continuous functions from $[0,T]$ into $W^{m,s}(\Omega)$, and $L^{p}(0,T;W^{m,s}(\Omega))$ for $1\le p\le \infty$ is similarly defined.

To obtain the stability of \eqref{hdg} in this section,  we assume  $ f_j \in C(0,T; L^2(\Omega)) $,  $g_j\in H^1(0,T; \\H^{1/2}(\partial \Omega))$, $u_j^0\in L^2(\Omega)$ and the vector fields $\bm{\beta}_j \in C(0,T; [W^{1,\infty}(\Omega)]^d)$ and  satisfying
\begin{align}\label{beta_con}
	\nabla\cdot\bm{\beta}_j \le 0.
\end{align}
These exists a positive constant $c_0$ such that the coefficients $c_j>c_0$,  and $c_j\in C(0,T; L^\infty(\Omega))$, and the ensemble mean satisfy the following condition
\begin{subequations}\label{cond-c}
	\begin{align}
		|c_{j}^n-\overline{c}^n|&<\frac 1 3 \min\{\overline{c}^n, \overline{c}^{n-1}, \overline{c}^{n-2}\} &n=2,3,\cdots,N,\\
		|c_{j}^1-\overline{c}^1|&< \min\{\overline{c}^1,\overline{c}^{0}\}.
	\end{align}
\end{subequations}

The following error estimates for the $ L^2 $ projections are standard:
\begin{lemma}\label{lemmainter}
	Suppose  integers $k, \ell \ge 0$. There exists a constant $C$ independent of $K\in\mathcal T_h$ such that
	\begin{subequations}
		\begin{align}
			&\|w - \Pi_{\ell}  w\|_K \le Ch^{\ell+1} |w|_{\ell+1,K}  &  &\forall w\in H^{\ell+1}(K), \label{lemmainter_orthoo}\\
			&\|w- P_M w\|_{\partial K} \le Ch^{k+1/2} |w|_{k+1,K}  &  &\forall w\in H^{k+1}(K). \label{lemmainter_orthoe}
		\end{align}
	\end{subequations}
\end{lemma}
We also use the following local inverse inequality:
\begin{align}\label{inverse_esti}
	\|w_h\|_{\partial K} &\le Ch_K^{-1/2} \|w_h\|_K& \forall w_h\in W_h.
\end{align}

\subsection{Preliminary material}
Next, we give the following several  lemmas, which will  be frequently used in our analysis.
\begin{lemma}\label{id}
	For any real numbers $a, b$ and $c$, we have 
	\begin{align*}
		\frac{1}{2}(3a-4b+c)a
		=\frac{1}{4}\left[ a^2+(2a-b)^2-b^2-(2b-c)^2\right]
		+\frac{1}{4}(a-2b+c)^2.
	\end{align*}
\end{lemma}

\begin{lemma}\label{beta}
	For $\bm\gamma\in [W^{1,\infty}(\Omega)]^d$ and $ w \in W_h$, we have 
	\begin{align}\label{betaa}
		(\bm \gamma w , \nabla  w)_{\mathcal T_h} = \frac 1 2 \langle  \bm{\gamma}\cdot \bm n  w,  w  \rangle_{\partial\mathcal T_h} - \frac 1 2 (\nabla \cdot\bm \gamma  w, w)_{\mathcal T_h}. 
	\end{align}
\end{lemma}
The proofs of \Cref{id} and \Cref{beta} are  trivial and we omit them here.
\begin{lemma}\label{error_time}
	Suppose the function $v:=v(\bm x,t)$ is smooth enough, then the following estimates hold true
	\begin{subequations}
		\begin{align}
			\|\partial^+_t v^n\|^2_{\mathcal{T}_h}
			&\le C\Delta t^{-1} \|\partial_{t} v\|^2_{[L^2(t_{n-2},t_n);L^2(\Omega)]},\label{es-parttial-01}\\
			\Delta t^4\|\partial^+_{tt} v^n\|^2_{\mathcal{T}_h}
			&\le C\Delta t^{3} \|\partial_{tt} v\|^2_{[L^2(t_{n-2},t_n);L^2(\Omega)]},\label{es-parttial-02}\\
			\|\partial_t v^n-\partial^+_t v^n\|^2_{\mathcal{T}_h}
			&\le C \Delta t^3 \|\partial_{ttt} v\|^2_{[L^2(t_{n-2},t_n);L^2(\Omega)]},\label{es-parttial-03}
		\end{align}
	\end{subequations}
	where $	\partial_{tt}^+v^n=(v^n-2v^{n-1}+v^{n-2})/\Delta t^2.$
\end{lemma}
The proof of \eqref{es-parttial-03} can be found in \cite{Gunzburger_Jiang_Wang_Flow_CMAM_2017}, the  proofs for \eqref{es-parttial-01}-\eqref{es-parttial-02} are very similar to the proof of \eqref{es-parttial-03} and hence we omit them. 

The following lemma is very crucial for our analysis.
\begin{lemma}\label{VV}
	For 
	$\bm \gamma\in [W^{1,\infty}(\Omega)]^d$,
	$(w,\mu)\in W_h\times M_h(0)$,
	$\nabla\cdot\bm \gamma\le 0$
	and $h$ small enough,  we have 
	\begin{align}\label{V1}
		\begin{split}
			\hspace{1em}&\hspace{-1em}\|h_K^{-1/2}(P_M w -\mu )\|^2_{\partial\mathcal{T}_h}
			-(\nabla\cdot\bm\gamma  w ,w )_{\mathcal{T}_h}-(\bm\gamma  w ,\nabla w )_{\mathcal{T}_h}
			+\langle{\bm \gamma} \cdot\bm n \mu , w  \rangle_{\partial\mathcal{T}_h}\\
			&\qquad\ge 
			\frac{1}{2}\|h_K^{-1/2}(P_M w -\mu )\|^2_{\partial\mathcal{T}_h}-Ch\|\nabla w \|^2_{\mathcal{T}_h}.
		\end{split}
	\end{align}
\end{lemma}
\begin{proof}
	Use  $\langle
	\bm\gamma \cdot\bm n \mu ,
	\mu  \rangle_{\partial\mathcal{T}_h}=0$,  $\nabla\cdot\bm \gamma  \le 0$ and integration by parts,  we have 
	\begin{align*}
		\hspace{1em}&\hspace{-1em}
		-(\nabla\cdot\bm\gamma  w ,w )_{\mathcal{T}_h}
		-(\bm\gamma  w ,\nabla w )_{\mathcal{T}_h}
		+\langle{\bm \gamma} \cdot\bm n \mu ,w  \rangle_{\partial\mathcal{T}_h}\\
		&=-\frac{1}{2} \langle
		\bm\gamma \cdot\bm n(w -\mu ), w -\mu  \rangle_{\partial\mathcal{T}_h}-\frac{1}{2}
		(\nabla\cdot\bm\gamma  w ,w )_{\mathcal{T}_h}&\textup{by} \ \eqref{betaa}\\
		&= -\frac{1}{2}\langle\bm\gamma \cdot\bm n(w -P_M w ),
		w - P_M w  \rangle_{\partial\mathcal{T}_h}-\langle \bm\gamma \cdot\bm n(P_M w -\mu ),
		(w - P_M w ) \rangle_{\partial\mathcal{T}_h}
		-\frac{1}{2} (\nabla\cdot\bm\gamma  w ,w )_{\mathcal{T}_h}\nonumber\\
		&\ge -C(h\|\nabla w \|^2_{\mathcal{T}_h}
		+h^{1/2}\|\nabla w \|_{\mathcal{T}_h}
		\|P_M w -\mu \|_{\partial\mathcal{T}_h})&\textup{by} \ \eqref{inverse_esti}\nonumber\\
		&\ge  -Ch\|\nabla w \|^2_{\mathcal{T}_h} -\frac{1}{4} \|h_K^{-1/2}(P_M w -\mu )\|_{\partial\mathcal{T}_h}^2.
	\end{align*}
	The mesh size $h$ small enough and 
	$\bm \gamma\in [W^{1,\infty}(\Omega)]^d$
	imply $\frac{1}{4}h_K^{-1}-\frac{1}{2}\bm{\gamma} \cdot\bm n\ge 0$, therefore, 
	\begin{align*}
		\hspace{1em}&\hspace{-1em}\|h_K^{-1/2}(P_M w -\mu )\|^2_{\partial\mathcal{T}_h}-(\nabla\cdot\bm\gamma  w ,w )_{\mathcal{T}_h}
		-(\bm\gamma  w ,\nabla w )_{\mathcal{T}_h}
		+\langle{\bm \gamma} \cdot\bm n\mu ,w  \rangle_{\partial\mathcal{T}_h}\nonumber\\
		&\ge 
		\frac{1}{2}\|h_K^{-1/2}(P_M w -\mu )\|^2_{\partial\mathcal{T}_h}-Ch\|\nabla w \|^2_{\mathcal{T}_h}.
	\end{align*}
\end{proof}

\begin{lemma}\label{ub}
	Let $(\bm q_{jh}^n, u_{jh}^n, \widehat{u}_{jh}^n)$ be the solution of \eqref{full_discretion_ensemble}, then we have the following bound
	\begin{align*}
		\|\nabla u_{jh}^n\|_{\mathcal{T}_h}\le
		C\left(
		\|\sqrt{\overline{c}^n}\bm q_{jh}^n\|_{\mathcal{T}_h}
		+
		\|\sqrt{\overline{c}^{n-1}}\bm q_{jh}^{n-1}\|_{\mathcal{T}_h}
		+	\|\sqrt{\overline{c}^{n-2}}\bm q_{jh}^{n-2}\|_{\mathcal{T}_h}+\|h_K^{-1/2}(P_M u_{jh}^n-\widehat{u}_{jh}^n)\|_{\partial\mathcal{T}_h}\right).
	\end{align*}
\end{lemma}
\begin{proof} 
	We take $\bm r_j=\nabla u_{jh}^n$ in \Cref{full_discretion_ensemble_a} and use 
	integration by parts to get
	\begin{align*}
		\|\nabla u_{jh}^n\|^2_{\mathcal{T}_h}&= -(\overline{c}^n\bm q^n_{jh},\nabla u_{jh}^n)_{\mathcal{T}_h}
		+\langle u_{jh}^n-\widehat{u}_{jh}^n,\nabla u_{jh}^n\cdot\bm n \rangle_{\partial\mathcal{T}_h}+((\overline c^n-c_j^n)(2\bm{q}^{n-1}_{jh}-\bm{q}^{n-2}_{jh}),\nabla u_{jh}^n)_{\mathcal{T}_h}\\
		&= -(\overline{c}^n\bm q^n_{jh},\nabla u_{jh}^n)_{\mathcal{T}_h}
		+\langle P_M u_{jh}^n-\widehat{u}_{jh}^n,\nabla u_{jh}^n\cdot\bm n \rangle_{\partial\mathcal{T}_h}
		+	( ( \overline c^n-c_j^n)(2\bm{q}^{n-1}_{jh}-\bm{q}^{n-2}_{jh}),\nabla u_{jh}^n)_{\mathcal{T}_h},
	\end{align*}
	then the desired result  is followed 
	by the Cauchy-Schwarz inequality and the local  inverse inequality  \eqref{inverse_esti}.
\end{proof}

\begin{lemma}[Discrete Poincar\'e-Friedrichs inequality]\label{poincare}
	For all  $(w,\mu) \in W_h\times M_h(0)$, we have
	\begin{equation*}
		\|w \|_{\mathcal T_h} \le   C  \|\nabla w \|_{\mathcal T_h}  + C \|h_K^{-1/2} (w  - \mu )\|_{\partial \mathcal T_h}.
	\end{equation*}
\end{lemma}
The proof of \Cref{poincare} is found in \cite[Lemma 5]{Chen_Monk_Peter1}.

\begin{lemma} 
	For all  $\bm\gamma \in [W^{1,\infty}({\Omega})]^d $ and  $(v ,w ,\widehat v ,\widehat w )\in W_h\times W_h\times M_h(0)\times M_h(0)$, we have 
	\begin{align}\label{frequ}
		\begin{split}
			\hspace{1em}&\hspace{-1em}-(\nabla\cdot\bm\gamma  w , v )_{\mathcal{T}_h}-(\bm\gamma w ,\nabla v )_{\mathcal{T}_h}
			+\langle \bm\gamma \cdot\bm n\widehat w , {v} \rangle_{\partial\mathcal{T}_h}\\
			&\le C\left(\|w \|_{\mathcal{T}_h}^2
			+\|v \|_{\mathcal{T}_h}^2+\|h_K^{1/2}(P_M v -\widehat{v} )\|_{\partial\mathcal{T}_h}^2\right)+(\nabla\cdot[\bm{\Pi}_0\bm\gamma  w ],v )_{\mathcal{T}_h}\\
			&\quad
			-\langle \bm{\Pi}_0\bm\gamma  \cdot\bm n w ,\widehat{v} \rangle_{\partial\mathcal{T}_h}+\langle \bm\gamma \cdot\bm n(\widehat w -w ), {v} -\widehat{v} \rangle_{\partial\mathcal{T}_h}.
		\end{split}
	\end{align}
\end{lemma}
\begin{proof} 
	We note that $\langle \bm \gamma  \cdot\bm n\widehat{w},  \widehat{v}  \rangle_{\partial\mathcal{T}_h}=0$, then 
	\begin{align*}
		\hspace{1em}&\hspace{-1em}-(\nabla\cdot\bm\gamma  w , v )_{\mathcal{T}_h}-(\bm\gamma w ,\nabla v )_{\mathcal{T}_h}
		+\langle \bm\gamma \cdot\bm n\widehat w , {v} \rangle_{\partial\mathcal{T}_h}\\
		&=(\bm\gamma \cdot\nabla w , v )_{\mathcal{T}_h}
		+\langle \bm\gamma \cdot\bm n(\widehat w -w ),  {v} \rangle_{\partial\mathcal{T}_h}&\textup{by} \ \eqref{betaa}\\
		&=(\bm\gamma \cdot \nabla w ,v )_{\mathcal{T}_h}
		-\langle \bm\gamma \cdot\bm nw , \widehat{v} \rangle_{\partial\mathcal{T}_h}
		+\langle \bm\gamma \cdot\bm n(\widehat w -w ),  {v} -\widehat{v} \rangle_{\partial\mathcal{T}_h}\\
		&=((\bm\gamma -\bm{\Pi}_0\bm\gamma )\cdot \nabla w ,v )_{\mathcal{T}_h}
		-\langle (\bm\gamma -\bm{\Pi}_0\bm\gamma )\cdot\bm nw ,  \widehat{v} \rangle_{\partial\mathcal{T}_h}\\
		&\quad +(\bm{\Pi}_0\bm\gamma \cdot \nabla w ,v )_{\mathcal{T}_h}
		-\langle (\bm{\Pi}_0\bm\gamma \cdot\bm nw , \widehat{v} \rangle_{\partial\mathcal{T}_h}+\langle \bm\gamma \cdot\bm n(\widehat w -w ), ({v} -\widehat{v} )\rangle_{\partial\mathcal{T}_h}.
	\end{align*}
	We use integration by parts  to get
	\begin{align*}
		\hspace{1em}&\hspace{-1em}-(\nabla\cdot\bm\gamma  w , v )_{\mathcal{T}_h}-(\bm\gamma w ,\nabla v )_{\mathcal{T}_h}
		+\langle \bm\gamma \cdot\bm n\widehat w , {v} \rangle_{\partial\mathcal{T}_h}\\
		&= - (\nabla\cdot(\bm\gamma  -\bm{\Pi}_0\bm \gamma ) w , v )_{\mathcal{T}_h}- ((\bm\gamma -\bm{\Pi}_0\bm\gamma )\cdot \nabla v ,w )_{\mathcal{T}_h}+\langle (\bm\gamma -\bm{\Pi}_0\bm\gamma )\cdot\bm nw ,  v -\widehat{v} \rangle_{\partial\mathcal{T}_h}\\
		&\quad+(\bm{\Pi}_0\bm\gamma \cdot \nabla w ,v )_{\mathcal{T}_h}
		-\langle (\bm{\Pi}_0\bm\gamma \cdot\bm nw , \widehat{v} \rangle_{\partial\mathcal{T}_h}+\langle \bm\gamma \cdot\bm n(\widehat w -w ),  	({v} -\widehat{v} )\rangle_{\partial\mathcal{T}_h}.
	\end{align*}
	Since $\bm \gamma  \in  [W^{1,\infty}({\Omega})]^d$, then $\|\bm\gamma -\bm{\Pi}_0\bm\gamma \|_{0,\infty,K}\le Ch_K\|\bm\gamma \|_{1,\infty,K}$.  Use the local  inverse inequality \eqref{inverse_esti} to get
	\begin{align*}
		\hspace{1em}&\hspace{-1em}-(\nabla\cdot\bm\gamma  w , v )_{\mathcal{T}_h}-(\bm\gamma w ,\nabla v )_{\mathcal{T}_h}
		+\langle \bm\gamma \cdot\bm n\widehat w , {v} \rangle_{\partial\mathcal{T}_h}\\
		&\le C\left(\|w \|_{\mathcal{T}_h}^2
		+\|v \|_{\mathcal{T}_h}^2+\|h_K^{1/2}(P_M v -\widehat{v} )\|_{\partial\mathcal{T}_h}^2\right)\\
		&\quad +(\nabla\cdot[\bm{\Pi}_0\bm\gamma  w ],v )_{\mathcal{T}_h}
		-\langle \bm{\Pi}_0\bm\gamma  \cdot\bm n w ,\widehat{v} \rangle_{\partial\mathcal{T}_h}+\langle \bm\gamma \cdot\bm n(\widehat w -w ), {v} -\widehat{v} \rangle_{\partial\mathcal{T}_h}.
	\end{align*}
	This proves the desired result.
\end{proof}

\subsection{Stability}
Next, we  prove the Ensemble HDG system \eqref{hdg} is  unconditionally stable.   Unlike the previous works, we do not assume the Dirichlet boundary conditions are zeros. Hence,  the proof here is more involved.

\begin{theorem}\label{stability_HDG} 
	The ensemble HDG system \eqref{hdg} is unconditionally stable if the condition \eqref{cond-c} holds. In particular, for $j=1,2,\ldots, J$, we have 
	\begin{align*}
		\hspace{1em}&\hspace{-1em}\max_{2\le n\le N}  \|u_{jh}^n\|_{\mathcal{T}_h}^2 +\Delta t \sum_{n=2}^{N} \|\sqrt{\overline{c}^n}\bm q_{jh}^n\|^2_{\mathcal{T}_h}\\
		&\le C\Delta t\sum_{n=2}^N\left(
		\|f_j^n\|^2_{\mathcal{T}_h}+\|g^n_j\|_{1/2,\partial\Omega}^2
		\right) + 
		C\left(\|u_{jh}^0\|^2_{\mathcal{T}_h}
		+\|u_{jh}^1\|^2_{\mathcal{T}_h}
		+\Delta t\|\sqrt{\overline{c}^1}\bm q_{jh}^1\|^2_{\mathcal{T}_h}
		+ \|\partial_{t} g_j\|^2_{L^2(0,T;H^{1/2}(\partial\Omega))}
		\right).
	\end{align*}
\end{theorem}

The proof of \Cref{stability_HDG} follows by triangle inequality, the definition of $H^{1/2}$ norm and \Cref{stability_HDG_lemma}.

To deal with the inhomogeneous boundary condition in the stability analysis, we need some additional notation. Let  $m_j\in H^1(0,T; H^1(\Omega))$ be an arbitrary function such that $m_j|_{\partial\Omega}=g_j$, and define 
\begin{align}\label{def_w}
	w_{jh}^n=u_{jh}^n-\Pi_{k+1} m_j^n,\quad \widehat{w}_{jh}^n=\widehat{u}_{jh}^n-P_M m_j^n.
\end{align}
This implies $ \widehat{w}_{jh}^n=0$ on $\mathcal{E}_h^{\partial}$. Now we give the estimate for $w_{jh}^n$.
\begin{lemma} 
	Let $(w_{jh}^n, \widehat{w}_{jh}^n)$ be defined in \eqref{def_w} and $(\bm q_{jh}^n, u_{jh}^n, \widehat{u}_{jh}^n)$ be the solution of \eqref{full_discretion_ensemble},	then we have the estimate
	\begin{align}\label{euqal-w}
		\begin{split}
			\|\nabla w_{jh}^n\|_{\mathcal{T}_h}
			& \le C\left(\|\sqrt{\overline{c}^n}\bm q_{jh}^n\|_{\mathcal{T}_h}
			+\|\sqrt{\overline{c}^{n-1}}\bm q_{jh}^{n-1}\|_{\mathcal{T}_h}
			+\|\sqrt{\overline{c}^{n-2}}\bm q_{jh}^{n-2}\|_{\mathcal{T}_h}\right)\\
			&\quad+C	\|h_K^{-1/2}(P_M w_{jh}^n-\widehat{w}_{jh}^n)\|_{\partial\mathcal{T}_h}
			+C\|\nabla m_j^n\|_{\mathcal{T}_h}.
		\end{split}
	\end{align}
	
\end{lemma}
\begin{proof}
	By  \Cref{ub} and  the triangle inequality,  we get
	\begin{align*}
		\|\nabla w_{jh}^n\|_{\mathcal{T}_h}&\le
		\|\nabla u_{jh}^n\|_{\mathcal{T}_h}+\|\nabla \Pi_{k+1} m_{j}^n\|_{\mathcal{T}_h}\\
		&\le	C\left(
		\|\sqrt{\overline{c}^n}\bm q_{jh}^n\|_{\mathcal{T}_h}+
		\|\sqrt{\overline{c}^{n-1}}\bm q_{jh}^{n-1}\|_{\mathcal{T}_h}
		+	\|\sqrt{\overline{c}^{n-2}}\bm q_{jh}^{n-2}\|_{\mathcal{T}_h} + \|h_K^{-1/2}(P_Mu_{jh}^n-\widehat{u}_{jh}^n)\|_{\partial\mathcal{T}_h}\right)\\
		&\quad  +C\|\nabla m_{j}^n\|_{\mathcal{T}_h}\\
		&\le C\left(
		\|\sqrt{\overline{c}^n}\bm q_{jh}^n\|_{\mathcal{T}_h}
		+
		\|\sqrt{\overline{c}^{n-1}}\bm q_{jh}^{n-1}\|_{\mathcal{T}_h}
		+	\|\sqrt{\overline{c}^{n-2}}\bm q_{jh}^{n-2}\|_{\mathcal{T}_h}\right)\\
		&\quad+C\left(\|h_K^{-1/2}(P_M w_{jh}^n-\widehat{w}_{jh}^n)\|_{\partial\mathcal{T}_h} +	\|h_K^{-1/2}(P_M\Pi_{k+1}m_j^n-P_M m_j^n)\|_{\partial\mathcal{T}_h} \right) +C\|\nabla m_{j}^n\|_{\mathcal{T}_h}\\
		&\le	C\left(
		\|\sqrt{\overline{c}^n}\bm q_{jh}^n\|_{\mathcal{T}_h}
		+
		\|\sqrt{\overline{c}^{n-1}}\bm q_{jh}^{n-1}\|_{\mathcal{T}_h}
		+	\|\sqrt{\overline{c}^{n-2}}\bm q_{jh}^{n-2}\|_{\mathcal{T}_h}\right)\nonumber\\
		&\quad+C	\|h_K^{-1/2}(P_M w_{jh}^n-\widehat{w}_{jh}^n)\|_{\partial\mathcal{T}_h}
		+C\|\nabla m_j^n\|_{\mathcal{T}_h}.
	\end{align*}
\end{proof}

\begin{lemma} \label{stability_HDG_lemma}
	Let $(w_{jh}^n, \widehat{w}_{jh}^n)$ be defined in \eqref{def_w} and $(\bm q_{jh}^n, u_{jh}^n, \widehat{u}_{jh}^n)$ be the solution of \eqref{full_discretion_ensemble}, if the condition \eqref{cond-c} holds, we have 
	\begin{align*}
		&\max_{2\le n\le N} \|w_{jh}^n\|_{\mathcal{T}_h}^2  +\Delta t\sum_{n=2}^N \|\sqrt{\overline{c}^n}\bm q_{jh}^n\|^2_{\mathcal{T}_h}
		\le
		C\Delta t\sum_{n=2}^N\left(
		\|f_j^n\|^2_{\mathcal{T}_h}+\|\nabla m_j^n\|_{\mathcal{T}_h}^2
		\right)\\
		&\qquad+C
		\left(\|w_{jh}^0\|^2_{\mathcal{T}_h}
		+
		\|w_{jh}^1\|^2_{\mathcal{T}_h}
		+ \|\partial_{t} m_j\|^2_{L^2(0,T;L^2(\Omega))}
		+\Delta t\|\sqrt{c^1}\bm q_{jh}^1\|^2_{\mathcal{T}_h}
		\right).
	\end{align*}
\end{lemma}

\begin{proof}  
	By the definitions of $w_{jh}^n$, $\widehat{w}_{jh}^n$ in \eqref{def_w},
	we can  rewrite \eqref{full_discretion_ensemble_a} and \eqref{full_discretion_ensemble_b} as
	\begin{subequations}\label{HDG2}
		\begin{align}\label{w-a}
			\begin{split}
				\hspace{1em}&\hspace{-1em} (\overline c^n\bm{q}^{n}_{jh},\bm{r}_j)_{\mathcal{T}_h}-(w^n_{jh},\nabla\cdot\bm{r}_j)_{\mathcal{T}_h}+\langle\widehat{w}^n_{jh},\bm{r}_j\cdot\bm n \rangle_{\partial{\mathcal{T}_h}}\\
				&=((\overline c^n-c_j^n)(2\bm{q}^{n-1}_{jh}-\bm{q}^{n-2}_{jh}),\bm{r}_j)_{\mathcal{T}_h} +(m_j^n,\nabla\cdot\bm r_j)_{\mathcal{T}_h}
				-\langle m_j^n,\bm{r}_j\cdot \bm n \rangle_{\partial\mathcal{T}_h}, 
			\end{split}
		\end{align}
		\begin{align}
			\begin{split}
				\hspace{1em}&\hspace{-1em} (\partial^+_tw^n_{jh},v_j)_{\mathcal T_h}+(\nabla\cdot\bm{q}^n_{jh}, v_j)_{\mathcal{T}_h}-\langle\bm{q}^n_{jh}\cdot \bm{n},\widehat{v}_j\rangle_{\partial{\mathcal{T}_h}}-(\nabla\cdot\overline{\bm\beta}^n w_{jh}^n,v_j)_{\mathcal{T}_h}\\
				&\quad -( \overline{\bm \beta}^n w_{jh}^{n}, \nabla v_j)_{\mathcal{T}_h}
				+\langle \overline{\bm \beta}^n\cdot\bm n, \widehat{w}_{jh}^{n} {v}_j\rangle_{\partial\mathcal{T}_h}+\langle h_K^{-1/2} ( P_M w^n_{jh}-{\widehat{w}}^n_{jh}),P_Mv_j-\widehat{v}_j \rangle_{\partial\mathcal{T}_h}\\
				&= (f^n_j,v_j)_{\mathcal{T}_h}
				-(\partial^+_t (\Pi_{k+1} w_{j}^n),v_j)_{\mathcal T_h}-([\nabla\cdot(\overline{\bm\beta}^n-\bm\beta_j^n)](2w_{jh}^{n-1}-w_{jh}^{n-2}),v_j)_{\mathcal{T}_h}\\
				&\quad-( (\overline{\bm \beta}^n-\bm{\beta}^n_j) (2w_{jh}^{n-1}
				-w_{jh}^{n-2}), \nabla v_j)_{\mathcal{T}_h}
				+\langle (\overline{\bm \beta}^n-\bm \beta_j^n)\cdot\bm n,(2\widehat{w}_{jh}^{n-1}-\widehat{w}_{jh}^{n-2}) {v}_j\rangle_{\partial\mathcal{T}_h}\\
				&\quad-([\nabla\cdot(\overline{\bm \beta}^n-\bm{\beta}^n_j)] (2\Pi_{k+1}m_{j}^{n-1}-\Pi_{k+1}m_{j}^{n-2}),v_j)_{\mathcal{T}_h}\\
				&\quad-((\overline{\bm \beta}^n-\bm{\beta}^n_j) (2\Pi_{k+1}m_{j}^{n-1}-\Pi_{k+1}m_{j}^{n-2}),\nabla v_j)_{\mathcal{T}_h}\\
				&\quad+\langle (\overline{\bm \beta}^n-\bm \beta_j^n)\cdot\bm n,(2P_M m_{j}^{n-1}- P_M m_{j}^{n-2}) {v}_j\rangle_{\partial\mathcal{T}_h}\\
				&\quad +(\nabla\cdot\overline{\bm\beta}^n \Pi_{k+1} m_j^{n},v_j)_{\mathcal{T}_h}+( \overline{\bm \beta}^n \Pi_{k+1} m_j^{n}, \nabla v_j)_{\mathcal{T}_h}
				-\langle \overline{\bm \beta}^n\cdot\bm n, P_M m_j^{n} {v}_j\rangle_{\partial\mathcal{T}_h}\\
				&\quad-\langle h_K^{-1}  P_M (\Pi_{k+1} m_j^n-m_j^n), P_M v_j-\widehat{v}_j \rangle_{\partial\mathcal{T}_h}.
			\end{split}
		\end{align}	
	\end{subequations}
	Now we take $(\bm r_j,v_j,\widehat{v}_j)=(\bm q_{jh}^n,w_{jh}^n,\widehat{w}_{jh}^n)$ in \eqref{HDG2}, add them together,  use
	\Cref{id} and stability  \eqref{V1} with $(w,\mu,\bm\gamma)=(w_{jh}^n,\widehat{w}_{jh}^n,\overline{\bm\beta}^n)$ to get
	\begin{align*}
		\hspace{1em}&\hspace{-1em}\frac{\|w_{jh}^n\|_{\mathcal{T}_h}^2+\|2w_{jh}^n-w_{jh}^{n-1}\|^2_{\mathcal{T}_h}-\|w_{jh}^{n-1}\|^2_{\mathcal{T}_h}-\|2w_{jh}^{n-1}-w_{jh}^{n-2}\|^2_{\mathcal{T}_h}}{4\Delta t}\\
		&\quad+\frac{\|w_{jh}^n-2w_{jh}^{n-1}+w_{jh}^{n-2}\|^2_{\mathcal{T}_h}}{4\Delta t}+\|\sqrt{\overline{c}^n}\bm q_{jh}^n\|^2_{\mathcal{T}_h}+\frac{1}{2}\|h_K^{-1/2}(P_M w_{jh}^n-\widehat{w}_{jh}^n)\|^2_{\partial\mathcal{T}_h}\\
		& \le  ((\overline c^n-c_j^n)(2\bm{q}^{n-1}_{jh}-\bm{q}^{n-2}_{jh}), \bm q_{jh}^n)_{\mathcal{T}_h} +(m_j^n,\nabla\cdot\bm q_{jh}^n)_{\mathcal{T}_h}-\langle m_j^n,\bm q_{jh}^n\cdot \bm n \rangle_{\partial\mathcal{T}_h}\\
		&\quad+ (f^n_j,w_{jh}^n)_{\mathcal{T}_h}
		-(\partial^+_t\Pi_{k+1} m_{j}^n,w_{jh}^n)_{\mathcal T_h}+Ch\|\nabla w_{jh}^n\|_{\mathcal{T}_h}^2\\
		&\quad
		-(\nabla\cdot(\overline{\bm\beta}^n-\bm\beta_j^n)(2w_{jh}^{n-1}-w_{jh}^{n-2}),w_{jh}^n)_{\mathcal{T}_h}-( (\overline{\bm \beta}^n-\bm{\beta}^n_j) (2w_{jh}^{n-1}-w_{jh}^{n-2}), \nabla w_{jh}^n)_{\mathcal{T}_h}\\
		&\quad+\langle (\overline{\bm \beta}^n-\bm \beta_j^n)\cdot\bm n,(2\widehat{w}_{jh}^{n-1}-\widehat{w}_{jh}^{n-2}) {w}_{jh}^n\rangle_{\partial\mathcal{T}_h}\\
		&\quad -( \nabla\cdot(\overline{\bm \beta}^n-\bm{\beta}^n_j) (2\Pi_{k+1} m_{j}^{n-1}-\Pi_{k+1} m_{j}^{n-2}
		), w_{jh}^n)_{\mathcal{T}_h}\\
		&\quad-( (\overline{\bm \beta}^n-\bm{\beta}^n_j) (2\Pi_{k+1} m_{j}^{n-1}-\Pi_{k+1}m_{j}^{n-2}
		),  \nabla  w_{jh}^n)_{\mathcal{T}_h}\\
		&\quad+\langle (\overline{\bm \beta}^n-\bm \beta_j^n)\cdot\bm n,(2 m_{j}^{n-1}-P_M m_{j}^{n-2}) w_{jh}^n\rangle_{\partial\mathcal{T}_h}\nonumber\\
		&\quad +(\nabla\cdot\overline{\bm\beta}^n \Pi_{k+1} m_j^{n},w_{jh}^n)_{\mathcal{T}_h}+( \overline{\bm \beta}^n \Pi_{k+1} m_j^{n}, \nabla w_{jh}^n)_{\mathcal{T}_h}
		-\langle \overline{\bm \beta}^n\cdot\bm n,P_M m_j^{n} w_{jh}^n\rangle_{\partial\mathcal{T}_h}\\
		&\quad
		-\langle h_K^{-1/2} P_M (P_M m_j^n-m_j^n),P_M w_{jh}^n-\widehat{w}_{jh}^n \rangle_{\partial\mathcal{T}_h}\\
		&=:\sum_{i=1}^{16}R_i.
	\end{align*}
	Next,  we estimate $\{R_i\}_{i=1}^{16}$ term by term.
	By \eqref{cond-c}, there exists a constant $\kappa>0$, such that
	\begin{align}
		|\overline{c}^n-c_{j}^n|\le\frac{\kappa}{3(\kappa+1)}\min\{\overline{c}^n,\overline{c}^{n-1},\overline{c}^{n-2}\}.\label{condition-cc}
	\end{align}
	Use  the above condition \eqref{condition-cc} and  the  Young's inequality to get
	\begin{align*}
		R_1&\le 
		\frac{\kappa}{3(\kappa+1)}\left(
		2\|\sqrt{\overline{c}^{n-1}}\bm q_{jh}^{n-1}\|_{\mathcal{T}_h}
		+\|\sqrt{\overline{c}^{n-2}}\bm q_{jh}^{n-2}\|_{\mathcal{T}_h}\right)\|\sqrt{\overline{c}^{n}}\bm q_{jh}^{n}\|_{\mathcal{T}_h}\\
		&\le
		\frac{\kappa}{3(\kappa+1)}
		\left(\frac{3}{2}\|\sqrt{\overline{c}^{n}}\bm q_{jh}^{n}\|_{\mathcal{T}_h}^2
		+\|\sqrt{\overline{c}^{n-1}}\bm q_{jh}^{n-1}\|_{\mathcal{T}_h}^2
		+\frac{1}{2}\|\sqrt{\overline{c}^{n-2}}\bm q_{jh}^{n-2}\|_{\mathcal{T}_h}^2
		\right).
	\end{align*}
	For the term $R_2+R_3$, we use integration by parts to get
	\begin{align*}
		R_2+R_3=-(\nabla m_j^n,\bm q_{jh}^n)\le \frac{1}{4(\kappa+1)}\|\sqrt{\overline{c}^n}\bm q_{jh}^n\|^2_{\mathcal{T}_h}
		+C\|\nabla m_j^n\|_{\mathcal{T}_h}^2.
	\end{align*}
	For the term $R_4$, we use the Cauchy-Schwarz inequality to get
	\begin{align*}
		R_4\le 2(\|w_{jh}^n\|^2_{\mathcal{T}_h}
		+\|f_j^n\|^2_{\mathcal{T}_h}).
	\end{align*}
	For the term $R_5$, by the definition of $\Pi_{k+1}$ in \eqref{L2_do} and  we use the Cauchy-Schwarz  inequality  and the estimate \eqref{es-parttial-01} to get
	\begin{align*}
		R_5&=-(\partial^+_t\Pi_{k+1} m_{j}^n,w_{jh}^n)_{\mathcal T_h}=-(\partial^+_t m_{j}^n,w_{jh}^n)_{\mathcal T_h}\\
		& \le C(\|\partial_t^+m_{j}^n\|^2_{\mathcal{T}_h}
		+ \|w_{jh}^n\|_{\mathcal{T}_h}^2) \le C\Delta t^{-1} \|\partial_{t} m_j^n\|^2_{L^2(t_{n-2}, t_n;L^2(\Omega))}
		+C\|w_{jh}^n\|_{\mathcal{T}_h}^2.
	\end{align*}
	For the term $R_6$, by the  estimate \eqref{euqal-w} and let $h$ sufficient small, one has
	\begin{align*}
		R_6&\le Ch\left(\|\sqrt{\overline{c}^n}\bm
		q_{jh}^n\|^2_{\mathcal{T}_h}+\|\sqrt{\overline{c}^{n-1}}\bm q_{jh}^{n-1}\|^2_{\mathcal{T}_h}+\|\sqrt{\overline{c}^{n-2}}\bm q_{jh}^{n-2}\|^2_{\mathcal{T}_h}\right)\\
		&\quad+Ch	\|h_K^{-1/2}(P_M w_{jh}^n-\widehat{w}_{jh}^n)\|^2_{\partial\mathcal{T}_h}
		+Ch\|\nabla m_j^n\|^2_{\mathcal{T}_h}\\
		&\le \frac{1}{24(\kappa+1)}\left(
		\|\sqrt{\overline{c}^n}\bm q_{jh}^n\|^2_{\mathcal{T}_h}+
		\|\sqrt{\overline{c}^{n-1}}\bm q_{jh}^{n-1}\|^2_{\mathcal{T}_h}
		+\|\sqrt{\overline{c}^{n-2}}\bm q_{jh}^{n-2}\|^2_{\mathcal{T}_h}\right)\\
		&\quad+\frac{1}{16}	\|h_K^{-1/2}(P_M w_{jh}^n-\widehat{w}_{jh}^n)\|^2_{\partial\mathcal{T}_h}
		+C\|\nabla m_j^n\|^2_{\mathcal{T}_h}.
	\end{align*}
	For the term $R_7+R_8+R_9$, let $(\bm\gamma,v ,w ,\widehat v ,\widehat w )
	=(\overline{\bm\beta}^n-\bm\beta_j^n,2w_{jh}^{n-1}-w_{jh}^{n-2},2\widehat w_{jh}^{n-1}-\widehat w_{jh}^{n-2},
	w_{jh}^n,\widehat w_{jh}^n)$ in  \eqref{frequ}  to get
	\begin{align*}
		R_7+R_8+R_9&\le
		C\left(\|w_{jh}^n\|_{\mathcal{T}_h}^2+\|w_{jh}^{n-1}\|_{\mathcal{T}_h}^2+\|w_{jh}^{n-2}\|_{\mathcal{T}_h}^2+\|h_K^{1/2}(P_M w_{jh}^{n}-\widehat{w}_{jh}^n)\|_{\partial\mathcal{T}_h}^2
		\right)\nonumber\\
		&\quad+(\nabla\cdot[\bm{\Pi}_0(\overline{\bm\beta}^n-\bm{\beta}_j^n) (2w_{jh}^{n-1}
		-w_{jh}^{n-2})],w_{jh}^n)_{\mathcal{T}_h}\nonumber\\
		&\quad-\langle[\bm{\Pi}_0^o(\overline{\bm\beta}^n-\bm{\beta}_j^n)(2w_{jh}^{n-1}-w_{jh}^{n-2})]\cdot\bm n,\widehat{w}_{jh}^n\rangle_{\partial\mathcal{T}_h}\nonumber\\
		&\quad+
		\langle 
		(\overline{\bm\beta}^n-\bm\beta_j^n)\cdot\bm n(2w_{jh}^{n-1}-w_{jh}^{n-2}-2\widehat{w}_{jh}^{n-1}+\widehat{w}_{jh}^{n-2}),  w_{jh}^n-\widehat{w}_{jh}^n
		\rangle_{\partial\mathcal{T}_h}.
	\end{align*}
	Using \eqref{w-a} with $\bm r_j=\bm{\Pi}_0 (\overline{\bm\beta}^n-\bm{\beta}_j^n)(2w_{jh}^{n-1}
	-w_{jh}^{n-2})\in\bm V_h$, we get
	\begin{align*}
		R_7+R_8+R_9&\le C\left(
		\|w_{jh}^n\|_{\mathcal{T}_h}^2+\|w_{jh}^{n-1}\|_{\mathcal{T}_h}^2
		+\|w_{jh}^{n-2}\|_{\mathcal{T}_h}^2+\|h_K^{1/2}(P_M w_{jh}^{n}-\widehat{w}_{jh}^n)\|_{\partial\mathcal{T}_h}^2
		\right)\nonumber\\
		&\quad+(\overline c^n\bm{q}^{n}_{jh},\bm{\Pi}_0 (\overline{\bm\beta}^n-\bm{\beta}_j^n) (2w_{jh}^{n-1}
		-w_{jh}^{n-2}))_{\mathcal{T}_h}\nonumber\\
		&\quad-((\overline c^n-c_j^n)(2\bm{q}^{n-1}_{jh}-\bm{q}^{n-2}_{jh}),\bm{\Pi}_0(\overline{\bm\beta}^n-\bm{\beta}_j^n) (2w_{jh}^{n-1}
		-w_{jh}^{n-2}))_{\mathcal{T}_h}\nonumber\\
		&\quad+(\nabla m_j^n,\bm{\Pi}_0(\overline{\bm\beta}^n-\bm{\beta}_j^n) (2w_{jh}^{n-1}-w_{jh}^{n-2}))_{\mathcal{T}_h}\nonumber\\
		&\quad+
		\langle (\overline{\bm\beta}^n-\bm\beta_j^n)\cdot\bm n(2w_{jh}^{n-1}-w_{jh}^{n-2}-2\widehat{w}_{jh}^{n-1}+\widehat{w}_{jh}^{n-2}), w_{jh}^n-\widehat{w}_{jh}^n
		\rangle_{\partial\mathcal{T}_h}.
	\end{align*}
	For $h$ small enough, Cauchy-Schwarz inequality and  Young's inequality give
	\begin{align*}
		R_7+R_8+R_9 &\le \frac{1}{24(\kappa+1)} \left(\|\sqrt{\overline{c}^n}\bm q_{jh}^n\|_{\mathcal{T}_h}^2+
		\|\sqrt{\overline{c}^{n-1}}\bm q_{jh}^{n-1}\|_{\mathcal{T}_h}^2
		+	\|\sqrt{\overline{c}^{n-2}}\bm q_{jh}^{n-2}\|_{\mathcal{T}_h}^2\right) \\
		&\quad +C\left(\|w_{jh}^n\|_{\mathcal{T}_h}^2
		+\|w_{jh}^{n-1}\|_{\mathcal{T}_h}^2
		+\|w_{jh}^{n-2}\|_{\mathcal{T}_h}^2
		+\|\nabla m_j^n\|_{\mathcal{T}_h}^2\right)\\
		&\quad +\frac{1}{16}	\|h_K^{-1/2}(P_M w_{jh}^n-\widehat{w}_{jh}^n)\|_{\partial\mathcal{T}_h}^2.
	\end{align*}
	Using integration by parts and the estimate \eqref{euqal-w}, we have 
	\begin{align*}
		\hspace{1em}&\hspace{-1em} R_{10}+R_{11}+R_{12}\\
		&\le C\left( \|w_{jh}^n\|_{\mathcal{T}_h}+\|\nabla w_{jh}^n\|_{\mathcal{T}_h} \right)
		\left(
		\|\nabla m_j^{n-1}\|_{\mathcal{T}_h}
		+	\|\nabla m_j^{n-2}\|_{\mathcal{T}_h}
		\right)
		\\
		&\le \alpha \left(\|\nabla w^n_{jh}\|_{\mathcal{T}_h}^2
		+\|h_K^{-1/2}(P_M w^n_{jh}-\widehat{w}^n_{jh}\|_{\partial\mathcal{T}_h}^2\right)
		+C_{\alpha}(\|\nabla m_j^{n-1}\|_{\mathcal{T}_h}^2+\|\nabla m_j^{n-2}\|_{\mathcal{T}_h}^2)\\
		&\le C\alpha \left(\|\sqrt{\overline{c}^n}\bm q_{jh}^n\|_{\mathcal{T}_h}^2+
		\|\sqrt{\overline{c}^{n-1}}\bm q_{jh}^{n-1}\|_{\mathcal{T}_h}^2
		+\|\sqrt{\overline{c}^{n-2}}\bm q_{jh}^{n-2}\|_{\mathcal{T}_h}^2\right)\\
		&\quad+C\alpha\left(	\|h_K^{-1/2}(P_Mw_{jh}^n-\widehat{w}_{jh}^n)\|_{\partial\mathcal{T}_h}^2+\|\nabla m_j^n\|_{\mathcal{T}_h}^2\right)\\
		&\quad +C_{\alpha}(\|\nabla m_j^{n-1}\|_{\mathcal{T}_h}^2+\|\nabla m_j^{n-2}\|_{\mathcal{T}_h}^2).
	\end{align*}
	Choose $\alpha$ small enough,  we get
	\begin{align*}
		\hspace{1em}&\hspace{-1em} R_{10}+R_{11}+R_{12}\\
		&\le \frac{1}{24(\kappa+1)} \left(\|\sqrt{\overline{c}^n}\bm q_{jh}^n\|^2_{\mathcal{T}_h}+\|\sqrt{\overline{c}^{n-1}}\bm q_{jh}^{n-1}\|^2_{\mathcal{T}_h}+\|\sqrt{\overline{c}^{n-2}}\bm q_{jh}^{n-2}\|^2_{\mathcal{T}_h}\right)\nonumber\\
		&\quad+\frac{1}{16}	\|h_K^{-1/2}(P_M w_{jh}^n-\widehat{w}_{jh}^n)\|^2_{\partial\mathcal{T}_h}
		+C\left(\|\nabla m_j^n\|^2_{\mathcal{T}_h}
		+\|\nabla m_j^{n-1}\|^2_{\mathcal{T}_h}
		+\|\nabla m_j^{n-2}\|^2_{\mathcal{T}_h}
		\right).
	\end{align*}
	Similarity,  we have 
	\begin{align*}
		\hspace{1em}&\hspace{-1em} R_{13}+R_{14}+R_{15}\\
		&\le \frac{1}{24(\kappa+1)}\left(\|\sqrt{\overline{c}^n}\bm q_{jh}^n\|^2_{\mathcal{T}_h}+\|\sqrt{\overline{c}^{n-1}}\bm q_{jh}^{n-1}\|^2_{\mathcal{T}_h}+\|\sqrt{\overline{c}^{n-2}}\bm q_{jh}^{n-2}\|^2_{\mathcal{T}_h}\right)\nonumber\\
		&\quad+\frac{1}{16}	\|h_K^{-1/2}(P_M w_{jh}^n-\widehat{w}_{jh}^n)\|^2_{\partial\mathcal{T}_h}
		+C\left(\|\nabla m_j^n\|^2_{\mathcal{T}_h}
		+\|\nabla m_j^{n-1}\|^2_{\mathcal{T}_h}
		+\|\nabla m_j^{n-2}\|^2_{\mathcal{T}_h}
		\right),\\
		\hspace{1em}&\hspace{-1em} R_{16} \le C\|\nabla m_j^n\|^2_{\mathcal{T}_h}+\frac{1}{16}	\|h_K^{-1/2}(P_M w_{jh}^n-\widehat{w}_{jh}^n)\|^2_{\partial\mathcal{T}_h}.
	\end{align*}

	Therefore, by all the estimate above one gets
	\begin{align*}
		&\quad \frac{\|w_{jh}^n\|_{\mathcal{T}_h}^2+\|2w_{jh}^n-w_{jh}^{n-1}\|^2_{\mathcal{T}_h}-\|w_{jh}^{n-1}\|^2_{\mathcal{T}_h}-\|2w_{jh}^{n-1}-w_{jh}^{n-2}\|^2_{\mathcal{T}_h}}{4\Delta t}
		\nonumber\\
		&\quad+\frac{\|w_{jh}^n-2w_{jh}^{n-1}+w_{jh}^{n-2}\|^2_{\mathcal{T}_h}}{4\Delta t}+\|\sqrt{\overline{c}^n}\bm q_{jh}^n\|^2_{\mathcal{T}_h}
		+\frac{1}{2}\|h_K^{-1/2}(P_M w_{jh}^n-\widehat{w}_{jh}^n)\|^2_{\partial\mathcal{T}_h}\nonumber\\
		& \le  C(\|w_{jh}^n\|^2_{\mathcal{T}_h}
		+\|w_{jh}^{n-1}\|^2_{\mathcal{T}_h}
		+\|w_{jh}^{n-2}\|^2_{\mathcal{T}_h}
		+\Delta t^{-1} \|\partial_{t} m_j^n\|^2_{L^2(t_{n-2},t_n;L^2(\Omega))}
		)\nonumber\\
		&\quad+C\left(
		\|f_j^n\|^2_{\mathcal{T}_h}+\|\nabla m_j^n\|_{\mathcal{T}_h}^2
		+\|\nabla m_j^{n-1}\|_{\mathcal{T}_h}^2
		+\|\nabla m_j^{n-2}\|_{\mathcal{T}_h}^2
		\right)\nonumber\\
		&\quad+
		\frac{2\kappa+1}{6(\kappa+1)}
		\left(
		\|\sqrt{c^n}\bm q_{jh}^n\|^2_{\mathcal{T}_h}
		+
		\|\sqrt{c^{n-1}}\bm q_{jh}^{n-1}\|^2_{\mathcal{T}_h}
		+	\|\sqrt{c^{n-2}}\bm q_{jh}^{n-2}\|^2_{\mathcal{T}_h}\right)\\
		&\quad +\frac{1}{4} \|h_K^{-1/2}(P_M w^n_{jh}-\widehat{w}^n_{jh})\|^2_{\partial\mathcal{T}_h}
		+\frac{1}{4(\kappa+1)}\|\sqrt{c^n}\bm q_{jh}^n\|^2_{\mathcal{T}_h}.
	\end{align*}
	We add last inequality from $n=2$ to $n=N$,
	rearrange it, and  multiply $4\Delta t$ to get
	\begin{align*}
		&\max_{2\le n\le N} \|w_{jh}^n\|_{\mathcal{T}_h}^2  +\Delta t\sum_{n=2}^N \|\sqrt{\overline{c}^n}\bm q_{jh}^n\|^2_{\mathcal{T}_h}
		\le
		C\Delta t\sum_{n=2}^N\left(
		\|f_j^n\|^2_{\mathcal{T}_h}+\|\nabla m_j^n\|_{\mathcal{T}_h}^2
		\right)\nonumber\\
		&\qquad+ C
		\left(\|w_{jh}^0\|^2_{\mathcal{T}_h}
		+
		\|w_{jh}^1\|^2_{\mathcal{T}_h}
		+ \|\partial_{t} m_j\|^2_{L^2(0,T;L^2(\Omega))}
		+\Delta t\|\sqrt{\overline{c}^1}\bm q_{jh}^1\|^2_{\mathcal{T}_h}
		\right).
	\end{align*}
	Then the result followed by Gronwall's inequality.
\end{proof}

\section{Error Analysis}
\label{errroanalysis}
The strategy of the  error analysis for the  Ensemble HDG method is based on \cite{ChenCockburnSinglerZhang1}. Throughout, we assume the data, the solutions of \eqref{convection_pde} are smooth enough and the domain $\Omega$ is convex.

\subsection{HDG elliptic projection}
\label{HDGellipticprojection}
For any $t\in[0,T]$ and $j=1,2,\ldots, J$,  let $(\overline{\bm q}_{jh},\overline{u}_{jh},\widehat{\overline u}_{jh})\in \bm V_h\times W_h\times M_h(g_j)$ be the solution of the following steady state problem
\begin{subequations}\label{projection}
	\begin{align}
		(c_j \overline{\bm{q}}_{jh},\bm{r}_j)_{\mathcal{T}_h}-(\overline u _{jh},\nabla\cdot \bm{r}_j)_{\mathcal{T}_h}+\langle\widehat{\overline u} _{jh},\bm{r}_j\cdot \bm n \rangle_{\partial{\mathcal{T}_h}} =0, \label{full_discretion_ensemble_steady1_a}
	\end{align}
	\begin{align}\label{full_discretion_ensemble_steady1_b}
		\begin{split}
			&\quad (\nabla\cdot\overline{\bm{q}} _{jh}, w_j)_{\mathcal{T}_h}-\langle \overline{\bm{q}} _{jh}\cdot \bm{n},\mu_j\rangle_{\partial{\mathcal{T}_h}}
			-( \nabla\cdot{\bm \beta}_j \overline{u}_{jh}, w_j)_{\mathcal{T}_h}-( {\bm \beta}_j \overline{u}_{jh},   \nabla w_j)_{\mathcal{T}_h}
			\\
			&\quad +\langle{\bm \beta}_j\cdot\bm n, \widehat{\overline{u}}_{jh} w_j\rangle_{\partial\mathcal{T}_h}+\langle h_K^{-1} (P_M \overline u _{jh}-{\widehat{\overline{u}}} _{jh}), P_M w_j-\mu_j \rangle_{\partial\mathcal{T}_h}\\
			&=(f _j-\partial_t u_j ,w_j)_{\mathcal{T}_h}
		\end{split}
	\end{align}
\end{subequations}
for all $(\bm r_j,w_j,\mu_j)\in \bm V_h\times W_h\times M_h(0)$.

The proofs of the following estimations are presented
in  \Cref{app}.
\begin{theorem}\label{error_dual}
	For any $t\in[0,T]$ and $j=1,2\cdots, J$, we have
	\begin{align*}
		\|\bm q_j- \overline{\bm q}_{jh}\|_{\mathcal{T}_h} \le  C h^{k+1},  \ \  \|\partial_t\bm q_j- \partial_t\overline{\bm q}_{jh}\|_{\mathcal{T}_h} \le C h^{k+1},  \ \  	\|\partial_{tt}\bm q_j- \partial_{tt}\overline{\bm q}_{jh}\|_{\mathcal{T}_h}  \le  C h^{k+1},\\
		\|u_j- \overline{u}_{jh}\|_{\mathcal{T}_h} \le  C h^{k+2},  \ \  \|\partial_t u_j- \partial_t\overline{u}_{jh}\|_{\mathcal{T}_h} \le C h^{k+2},  \ \  	\|\partial_{tt} u_j- \partial_{tt}\overline{u}_{jh}\|_{\mathcal{T}_h}  \le  C h^{k+2},
	\end{align*}
	and 
	\begin{align*}
		\|h_K^{1/2}( \overline{u}_{jh}-\widehat{\overline u}_{jh})\|_{\partial\mathcal{T}_h} \le Ch^{k+1},\quad \|h_K^{1/2}( \partial_{tt}\overline{u}_{jh}-\partial_{tt}\widehat{\overline u}_{jh})\|_{\partial\mathcal{T}_h} \le Ch^{k+1}.
	\end{align*}
\end{theorem}

\subsection{Main result}
We can now state our main result for the ensemble HDG method.
\begin{theorem}\label{main_theorem}
	If  the condition \eqref{cond-c} holds and the domain is convex, then we have the following error estimate
	\begin{align}\label{01}
		\max_{1\le n\le N}
		\|u^n_j-u^n_{jh}\|_{\mathcal{T}_h}\le C\left(\Delta t^2+h^{k+2}\right),\\
		\sqrt{\Delta t\sum_{n=1}^N\|\sqrt{\overline c^n}(\bm q_j^n-\bm q_{jh}^n)\|^2_{\mathcal{T}_h}}\le C\left(\Delta t^2+h^{k+1}\right).
	\end{align}
\end{theorem}
\begin{remark} 
	To the best of our knowledge, all previous works only contain a \emph{suboptimal} $L^\infty(0,T;\\L^2(\Omega))$ convergent rate for the ensemble solutions $u_j$.  Only one other very recent work \cite{PiChenZhangXu1} contains  an \emph{optimal} $L^\infty(0,T;L^2(\Omega))$ convergent rate for the ensemble solutions $u_j$,  and  a $L^2(0,T;L^2(\Omega))$  superconvergent rate if the coefficients of the PDEs are independent of time and degree polynomial $k\ge 1$;  our main result: \Cref{main_theorem} is the \emph{first} time to obtain the  $L^\infty(0,T;L^2(\Omega))$ {supconvergent} rate for the ensemble solutions $u_j$ for all $k\ge 0$ and  without assume that the coefficients of the PDEs are independent of time. It is also the \emph{first} time to obtain the superonvergent rate for a single convection diffusion PDE when $k=0$.
\end{remark}

\subsection{Proof of \Cref{main_theorem}}
The proof of \eqref{01}  with $n=1$ is quite standard in backward Euler discretization, thus we omit it, and we prove \eqref{01} holds for all $n\geq 2$.
\subsubsection{The equations of the projection of the errors}
\begin{lemma} \label{main_error_equations1}
	For  $ e^{\bm q^n}_{jh}=\bm q_{jh}^n-{\overline{\bm q}}_{jh}^n, \ \  e^{u^n}_{jh}=u_{jh}^n-{\overline{u}}_{jh}^n, \ \  
	e^{\widehat{u}^n}_{jh}=\widehat{u}_{jh}^n-{{\widehat{\overline{u}}}}_{jh}^n$, for all $j = 1,2\cdots, J$, we have  the following error equations
	\begin{subequations}\label{error1}
		\begin{align}
			\begin{split}
				(\overline c^ne_{jh}^{\bm q^n},\bm{r}_j)_{\mathcal{T}_h}
				-(e^{u^n}_{jh},\nabla\cdot \bm{r}_j)_{\mathcal{T}_h}+\langle e^{\widehat{u}^n}_{jh},\bm{r}_j\cdot \bm n \rangle_{\partial{\mathcal{T}_h}} =
				( ( \overline c^n-c_j^n)(2\bm{q}^{n-1}_{jh}-\bm{q}^{n-2}_{jh}-
				\overline{\bm{q}}^{n}_{jh}),\bm{r}_j)_{\mathcal{T}_h},
				\label{error_a}
			\end{split}
		\end{align}
		\begin{align}\label{errorb}
			\begin{split}
				\hspace{1em}&\hspace{-1em} (\partial^+_te^{u^n}_{jh},w_j)_{\mathcal T_h}+(\nabla\cdot e_{jh}^{\bm q^n}, w_j)_{\mathcal{T}_h}-\langle e^{\bm{q}^n}_{jh}\cdot \bm{n},\mu_j\rangle_{\partial{\mathcal{T}_h}}
				-(\nabla\cdot \overline{\bm \beta}^n e^{u^n}_{jh}, w_j)_{\mathcal{T}_h}\\
				& \quad -( \overline{\bm \beta}^ne^{u^n}_{jh}, \nabla w_j)_{\mathcal{T}_h}+\langle \overline{\bm \beta}^n\cdot\bm n, e^{{\widehat u}^n}_{jh} w_j\rangle_{\partial\mathcal{T}_h} +\langle h_K^{-1} ( P_M e^{u^n}_{jh}-e^{\widehat{u}^n}_{jh}),
				P_M w_j-\mu_j \rangle_{\partial\mathcal{T}_h}\\
				&=-( \nabla\cdot(\overline{\bm \beta}^n-\bm{\beta}_j^n) 
				(2{u}_{jh}^{n-1}-{u}_{jh}^{n-2}-\overline{u}^n_{jh})
				, w_j)_{\mathcal{T}_h} -( (\overline{\bm \beta}^n-\bm{\beta}_j^n) 
				(2{u}_{jh}^{n-1}-{u}_{jh}^{n-2}-\overline{u}^n_{jh})
				, \nabla w_j)_{\mathcal{T}_h}\\
				&\quad+\langle (\overline{\bm \beta}^n_j-\bm{\beta}_j^n)\cdot\bm n, (2{	\widehat u}_{jh}^{n-1}-{\widehat u}_{jh}^{n-2}-	\widehat{\overline{ u}}^n_{jh})
				w_j\rangle_{\partial\mathcal{T}_h}
				+(\partial_tu_j^n-\partial_t^+\overline{u}^n_{jh},w_j)_{\mathcal{T}_h}
			\end{split}
		\end{align}
	\end{subequations}
	for all $(\bm r_j,w_j,\mu_j)\in \bm V_h\times W_h\times M_h(0)$ and $n=1,2,\cdots,N$.
\end{lemma}
The proof of \Cref{main_error_equations1} follows  by  subtracting \Cref{projection} from \Cref{hdg}.

\subsubsection{Energy argument}

We take $\bm r_j=\nabla e_{jh}^{u^n}$ in \eqref{error_a} and use integration by parts to get the following lemma.
\begin{lemma}
	We have 
	\begin{align}\label{equal-uq}
		\begin{split}
			\|\nabla e_{jh}^{u^n}\|_{\mathcal{T}_h}+\|h_K^{-1/2} (e_{jh}^{u^n}-e_{jh}^{\widehat{u}^n})\|_{\partial\mathcal{T}_h} &\le
			C\left( \|\sqrt{\overline c^n}e_{jh}^{\bm q^n}\|_{\mathcal{T}_h}
			+\|h_K^{-1/2}(P_M e_{jh}^{u^n}-e_{jh}^{\widehat{u}^n})\|_{\partial\mathcal{T}_h}\right)\\
			&\quad +C\|( \overline c^n-c_j^n)(2\bm{q}^{n-1}_{jh}-\bm{q}^{n-2}_{jh}-
			\overline{\bm{q}}^{n}_{jh})\|_{\mathcal{T}_h}.
		\end{split}
	\end{align}
\end{lemma}

\begin{lemma}\label{intermed}
	If  the condition \eqref{cond-c} holds and the domain is convex,  then we have the following error estimate
	\begin{align*}
		\max_{2\le n\le N} 
		\|e^{{u}^n}_{jh}\|_{\mathcal{T}_h} + \sqrt{\Delta t\sum_{n=2}^N \|\sqrt{\overline c^n}e^{\bm{q}^{n}}_{jh}\|^2_{\mathcal{T}_h}} \le C\left(\Delta t^2+h^{k+2}\right).
	\end{align*}	
\end{lemma}
\begin{proof}
	We take $(\bm r_j,w_j,\mu_j)=(e^{\bm{q}^{n}}_{jh},e^{u^n}_{jh},e^{{\widehat{u}}^n}_{jh} )$
	in \eqref{error1}, use the polarization
	identity \eqref{id},
	stability \eqref{V1} with $(\bm\gamma,w,\mu)=(\bm\beta_j^n,e_{jh}^{u^n},e_{jh}^{\widehat u^n})$,
	and add them together to get
	\begin{align}
		&\quad \frac{\|e^{{u}^n}_{jh}\|^2_{\mathcal{T}_h}+\|2e^{{u}^n}_{jh}-e^{{u}^{n-1}}_{jh}\|^2_{\mathcal{T}_h}-\|e^{{u}^{n-1}}_{jh}\|^2_{\mathcal{T}_h}-\|2e^{{u}^{n-1}}_{jh}-e^{{u}^{n-2}}_{jh}\|^2_{\mathcal{T}_h}}{4\Delta t}\nonumber\\
		&\quad +\frac{\|e^{{u}^n}_{jh}-2e^{{u}^{n-1}}_{jh}+e^{{u}^{n-2}}_{jh}\|^2_{\mathcal{T}_h}}{4\Delta t}+\|\sqrt{\overline c^n}e^{\bm{q}^{n}}_{jh}\|^2_{\mathcal{T}_h}
		+\frac{1}{2}\|h_K^{-1/2}(P_M e^{u^n}_{jh}-e^{{\widehat{u}}^n}_{jh})\|^2_{\partial\mathcal{T}_h}\nonumber\\
		&\le
		((\overline c^n-c^n_j)(2\bm{q}^{n-1}_{jh}-\bm{q}^{n-2}_{jh}-
		\overline{\bm{q}}^{n}_{jh}),e^{\bm{q}^{n}}_{jh} )+
		(\partial_t u_j^n-\partial^+_t\overline{u}_{jh}^n,e^{u^n}_{jh})_{\mathcal{T}_h}\nonumber\\
		&\quad
		-( \nabla\cdot(\overline{\bm \beta}^n-\bm{\beta}_j^n) (
		2{u}_{jh}^{n-1}-{u}_{jh}^{n-2}
		-\overline{u}_{jh}^{n}), e^{u^n}_{jh})_{\mathcal{T}_h}\nonumber\\
		&\quad-( (\overline{\bm \beta}^n-\bm{\beta}_j^n) (
		2{u}_{jh}^{n-1}-{u}_{jh}^{n-2}
		-\overline{u}_{jh}^{n}), \nabla e^{u^n}_{jh})_{\mathcal{T}_h}\nonumber\\
		&\quad+\langle (\overline{\bm \beta}^n-\bm{\beta}_j^n)\cdot\bm n, (2{\widehat u}_{jh}^{n-1}-{\widehat u}_{jh}^{n-2}-\widehat{\overline u}^n_{jh}) e^{{u}^n}_{jh}\rangle_{\partial\mathcal{T}_h}+Ch\|\nabla\varepsilon_{jh}^{u^n}\|^2_{\mathcal{T}_h}\nonumber\\
		&=:\sum_{i=1}^6 R_i.\label{29}
	\end{align}
	Next,  we estimate $\{R_i\}_{i=1}^6$ term by term. For the first term $R_1$, since 
	\begin{align}\label{identity_q}
		2\bm q_{jh}^{n-1}-\bm q_{jh}^{n-2}-\overline{\bm q}_{jh}^n=
		2e_{jh}^{\bm q^{n-1}}-e_{jh}^{\bm q^{n-2}}-\Delta t^2\partial_{tt}^+\overline{\bm q}_{jh}^n,
	\end{align}
	we use condition \eqref{condition-cc} to get 
	\begin{align*}
		R_1&=((\overline c^n-c^n_j)(
		2e_{jh}^{\bm q^{n-1}}-e_{jh}^{\bm q^{n-2}}
		-\Delta t^2\partial_{tt}^+\overline{\bm q}_{jh}^n),e^{\bm{q}^{n}}_{jh})\\
		&\le \frac{\kappa}{3(\kappa+1)}\left(\frac{3}{2}\|\sqrt{\overline c^n}e^{\bm{q}^{n}}_{jh}\|^2_{\mathcal{T}_h}
		+\|\sqrt{\overline c^{n-1}}e^{\bm{q}^{n-1}}_{jh}\|_{\mathcal T_h}^2
		+\frac{1}{2}\|\sqrt{\overline c^{n-1}}e^{\bm{q}^{n-2}}_{jh}\|^2_{\mathcal{T}_h}
		\right)\nonumber\\
		&\quad +\frac{1}{8(\kappa+1)}\|\sqrt{\overline c^n}e^{\bm{q}^{n}}_{jh}\|^2_{\mathcal{T}_h} 
		+C\Delta t^4\|\partial_{tt}^+\overline{\bm q}_{jh}^n\|^2_{\mathcal{T}_h}.
	\end{align*}
	For the term $R_2$, we have 
	\begin{align*}
		R_2&=(\partial^+_t(u^n_j-\overline{u}_{jh}^n)
		+\partial_tu_j^n-\partial^+_tu^n_j ,e^{u^n}_{jh})_{\mathcal{T}_h}\\
		&\le C\left( \|\partial^+_t(u^n_j-\overline{u}_{jh}^n)\|^2_{\mathcal{T}_h}
		+\|\partial_tu_j^n-\partial^+_tu^n_j\|^2_{\mathcal{T}_h}
		+\|e^{u^n}_{jh}\|^2_{\mathcal{T}_h} \right).
	\end{align*}
	For the term $R_3+R_4+R_5$,  \Cref{frequ} and  \eqref{error_a} give 
	\begin{align*}
		\hspace{1em}&\hspace{-1em} R_3+R_4+R_5\\
		&\le C\left(\|	2{u}_{jh}^{n-1}-{u}_{jh}^{n-2}
		-\overline{u}_{jh}^{n}\|_{\mathcal{T}_h}^2
		+\|e_{jh}^{u^n}\|_{\mathcal{T}_h}^2+\|h_K^{1/2}(P_M e_{jh}^{u^n}-{e}_{jh}^{\widehat{u}^n})\|_{\partial\mathcal{T}_h}^2
		\right)\\
		&\quad+\left(\nabla\cdot\left[\bm{\Pi}_0(\overline{\bm \beta}^n-\bm{\beta}^n_j) (2{u}_{jh}^{n-1}-{u}_{jh}^{n-2}-\overline{u}_{jh}^{n})\right],e_{jh}^{u^n}\right)_{\mathcal{T}_h}\\
		&\quad
		-\langle \bm{\Pi}_0(\overline{\bm \beta}^n-\bm{\beta}^n_j)(2{u}_{jh}^{n-1}-{u}_{jh}^{n-2}-\overline{u}_{jh}^{n})\cdot\bm n,e_{jh}^{\widehat{u}^n}\rangle_{\partial\mathcal{T}_h}\\
		&\quad+ \langle
		(\overline{\bm\beta}^n-\bm\beta_j^n)\cdot\bm n
		(2u_{jh}^{n-1}-u_{jh}^{n-2}-\overline{u}_{jh}^n
		-2\widehat{u}_{jh}^{n-1}+\widehat{u}_{jh}^{n-2}+\widehat{\overline{u}}_{jh}^n
		),e_{jh}^{u^n}-e_{jh}^{\widehat u^n}
		\rangle_{\partial\mathcal{T}_h}\\
		& =C\left(\|2{u}_{jh}^{n-1}-{u}_{jh}^{n-2}
		-\overline{u}_{jh}^{n}\|_{\mathcal{T}_h}^2
		+\|e_{jh}^{u^n}\|_{\mathcal{T}_h}^2
		+\|h_K^{1/2}(P_M e_{jh}^{u^n}-{e}_{jh}^{\widehat{u}^n})\|_{\partial\mathcal{T}_h}^2
		\right)\\
		&\quad+	(\overline c^ne_{jh}^{\bm q^n},\bm{\Pi}_0(\overline{\bm \beta}^n-\bm{\beta}^n_j)(2{u}_{jh}^{n-1}-{u}_{jh}^{n-2}-\overline{u}_{jh}^{n}))_{\mathcal{T}_h}\\
		&\quad-
		( ( \overline c^n-c_j^n)(2\bm{q}^{n-1}_{jh}-\bm{q}^{n-2}_{jh}-
		\overline{\bm{q}}^{n}_{jh}),\bm{\Pi}_0(\overline{\bm \beta}^n-\bm{\beta}^n_j)(2{u}_{jh}^{n-1}-{u}_{jh}^{n-2}-\overline{u}_{jh}^{n}))_{\mathcal{T}_h}\\
		&\quad+ \langle
		(\overline{\bm\beta}^n-\bm\beta_j^n)\cdot\bm n
		(2u_{jh}^{n-1}-u_{jh}^{n-2}-\overline{u}_{jh}^n
		-2\widehat{u}_{jh}^{n-1}+\widehat{u}_{jh}^{n-2}+\widehat{\overline{u}}_{jh}^n),e_{jh}^{u^n}-e_{jh}^{\widehat  u^n}
		\rangle_{\partial\mathcal{T}_h}.
		&\end{align*}
	Similar  to \eqref{identity_q}, we have 
	\begin{align*}
		2u_{jh}^{n-1}-u_{jh}^{n-2}-\overline{u}_{jh}^n=
		2e_{jh}^{u^{n-1}}-e_{jh}^{u^{n-2}}-\Delta t^2\partial_{tt}^+\overline{u}_{jh}^n,
	\end{align*}
	and 
	\begin{align*}
		\hspace{-1em}&\hspace{1em} 2u_{jh}^{n-1}-u_{jh}^{n-2}-\overline{u}_{jh}^n
		-2\widehat{u}_{jh}^{n-1}+\widehat{u}_{jh}^{n-2}+\widehat{\overline{u}}_{jh}^n \\
		&=2(e_{jh}^{u^{n-1}}-e_{jh}^{\widehat u^{n-1}})
		-(e_{jh}^{u^{n-2}}-e_{jh}^{\widehat u^{n-2}})
		-\Delta t^2\partial^+_{tt}\overline{u}_{jh}^n
		+\Delta t^2\partial^+_{tt}{\widehat {\overline u}}_{jh}^n.
	\end{align*}
	Therefore, when $h$ is small enough,
	we have
	\begin{align*}
		\hspace{1em}&\hspace{-1em} R_3+R_4+R_5\\
		&\le \frac{1}{24(\kappa+1)}\left(\|\sqrt{\overline c^n}e^{\bm{q}^{n}}_{jh}\|^2_{\mathcal{T}_h}
		+\|\sqrt{\overline c^{n-1}}e^{\bm{q}^{n-1}}_{jh}\|^2_{\mathcal{T}_h}
		+\|\sqrt{\overline c^{n-2}}e^{\bm{q}^{n-2}}_{jh}\|^2_{\mathcal{T}_h}\right)\nonumber\\
		&\quad +\frac{1}{16}\|h_K^{-1/2}(P_M e^{u^n}_{jh}-e^{{\widehat{u}}^n}_{jh})\|^2_{\mathcal{T}_h}+
		C\left(
		\|e^{u^n}_{jh}\|^2_{\mathcal{T}_h}
		+\|e^{u^{n-1}}_{jh}\|^2_{\mathcal{T}_h}
		+\|e^{u^{n-2}}_{jh}\|^2_{\mathcal{T}_h}	
		\right)\\
		&\quad +\frac{1}{16}\|h_K^{-1/2}(P_M e^{u^{n-1}}_{jh}-e^{{\widehat{u}}^{n-1}}_{jh})\|^2_{\mathcal{T}_h}
		+\frac{1}{16}\|h_K^{-1/2}(P_M e^{u^{n-2}}_{jh}-e^{{\widehat{u}}^{n-2}}_{jh})\|^2_{\mathcal{T}_h}
		\\
		&\quad +C\Delta t^4\|\partial_{tt}^+\overline{\bm q}_{jh}^n\|^2_{\mathcal{T}_h}+C\Delta t^4\|\partial_{tt}^+{\overline{u}_{jh}^n}\|^2_{\mathcal{T}_h}
		+C\Delta t^4\|h_K^{1/2} \partial_{tt}^+({\overline{u}_{jh}^n}
		-\widehat{\overline{u}}_{jh}^n)
		\|^2_{\partial\mathcal{T}_h}.
	\end{align*}
	By the Cauchy-Schwarz inequality and $h$ small enough, by \eqref{equal-uq} and \eqref{identity_q}, we  get
	\begin{align*}
		R_6
		&\le \frac{1}{24(\kappa+1)} \left(\|\sqrt{\overline c^n}e^{\bm{q}^{n}}_{jh}\|^2_{\mathcal{T}_h}
		+\|\sqrt{\overline c^{n-1}}e^{\bm{q}^{n-1}}_{jh}
		+\|\sqrt{\overline c^{n-1}}e^{\bm{q}^{n-2}}_{jh}\|^2_{\mathcal{T}_h}
		\right)+C\Delta t^4\|\partial_{tt}^+\overline{\bm q}_{jh}^n\|^2_{\mathcal{T}_h}.
	\end{align*}
	We add \eqref{29} from $n=2$ to $n=N$ 
	and use the above  inequalities  to get
	\begin{align}\label{214}
		\begin{split}
			\hspace{0.1em}&\hspace{-0.1em}\max_{2\le n\le N} 
			\|e^{{u}^n}_{jh}\|^2_{\mathcal{T}_h} +\Delta t\sum_{n=2}^N \|\sqrt{\overline c^n}e^{\bm{q}^{n}}_{jh}\|^2_{\mathcal{T}_h} \\
			&\le C	\Delta t^5\sum_{n=2}^N(
			\|\partial_{tt}^+\overline{u}_{jh}^n\|^2_{\mathcal{T}_h}
			+\|\partial_{tt}^+\overline{\bm q}_{jh}^n\|^2_{\mathcal{T}_h}
			+\|h_K^{1/2}\partial_{tt}^+({\overline{u}_{jh}^n}
			-\widehat{\overline{u}}_{jh}^n)
			\|^2_{\partial\mathcal{T}_h}),\\
			&\quad +C\Delta t\sum_{n=2}^N
			(
			\|\partial^+_t( u_j^n-\overline{u}_{jh}^n)\|^2_{\mathcal{T}_h} +\|\partial_t u_j^n-\partial^+_t u_j^n\|^2_{\mathcal{T}_h}
			) + C\Delta t\sum_{n=2}^N \|e^{u^n}_{jh}\|^2_{\mathcal{T}_h}\\
			&\quad+ C\left(
			\|e_{jh}^{u^0}\|^2_{\mathcal{T}_h}
			+\|e_{jh}^{u^1}\|^2_{\mathcal{T}_h}
			+\Delta t\|e_{jh}^{\bm q^1}\|^2_{\mathcal{T}_h}
			\right).
		\end{split}
	\end{align}
	Next,  
	we  bound the terms on the right side of \eqref{214} by  \Cref{error_time}.
	\begin{align*}
		\Delta t^5\sum_{n=2}^N\|\partial_{tt}^+\overline{u}_{jh}^n\|^2_{\mathcal{T}_h} &\le C\Delta t^4 \|\partial_{tt}\overline{u}_{jh}\|^2_{L^2(0,T;L^2(\Omega))},\\
		\Delta t^5\sum_{n=2}^N\|\partial_{tt}^+\overline{\bm q}_{jh}^n\|^2_{\mathcal{T}_h}
		&\le C\Delta t^4 \|\partial_{tt}\overline{\bm q}_{jh}\|^2_{_{L^2(0,T;L^2(\Omega))}},\\
		\Delta t
		\sum_{n=2}^N
		\|\partial^+_t( u_j^n-\overline{u}_{jh}^n)\|^2_{\mathcal{T}_h}
		&\le C\|\partial_t(u_j-\overline{u}_{jh})\|_{L^2(0,T; L^2(\Omega))}^2,\\
		\Delta t\sum_{n=2}^N\|\partial_t u_j^n-\partial^+_t u_j^n\|^2_{\mathcal{T}_h}
		&\le C\Delta t^4 \|\partial_{ttt} u_j\|^2_{L^2(0,T; L^2(\Omega))},\\
		\Delta t^5\sum_{n=2}^N\|h_K^{1/2}\partial_{tt}^+({\overline{u}_{jh}^n}
		-\widehat{\overline{u}}_{jh}^n)
		\|^2_{\partial\mathcal{T}_h}
		&\le C\Delta t^4 \| h_K^{1/2} \partial_{tt}
		({\overline{u}_{jh}}
		-\widehat{\overline{u}}_{jh})
		\|_{L^2(0,T;L^2(\partial\mathcal{T}_h))}^2.
	\end{align*}
	Gronwall's inequality,  the estimates above, \Cref{error_dual}
	applied to \eqref{214}, and \eqref{01}  give the desired result.
\end{proof}

As a consequence, a simple application of the triangle inequality for \eqref{intermed} and \Cref{error_dual} give the proof of \Cref{main_theorem}.

\section{Numerical experiments}
\label{numericalexperiments}
In this section, we present some numerical tests of the Ensemble HDG method  for parameterized convection diffusion PDEs. A group of simulations are  considered containing  $J = 3$ members. Let $Eu_j$ be the error bewteen the exact solution $u_j$ at the final time $T=1$ and the Ensemble HDG solution $u_{jh}^N$, i.e., $Eu_j = \|u_j^N -u_{jh}^N\|_{\mathcal T_h}$. Let 
\begin{align*}
	E\bm q _j = \sqrt{\Delta t\sum_{n=1}^N \|\bm q_j^n - \bm q_{jh}^n\|^2_{\mathcal T_h}}.
\end{align*}
We  test the convergence rate of the Ensemble HDG method for on a square domain $\Omega = [0,1]\times[0,1]$. The data is chosen as
\begin{gather*}
	c_1 = 1.1, \ c_2 = 1.2, \ c_3 = 1.3,  \ \bm \beta_1 = [1,1], \ \bm \beta_2 = [2,2], \ \bm \beta_3 = [3,3], \\
	u_1= e^{-t}\sin(x), \ u_2= \cos (t)\cos(x), \ u_3= e^{x-t}, 
\end{gather*}
and the initial conditions, boundary conditions, and source terms are chosen to match the exact solution of  \Cref{convection_pde}. It is easy to see that the  coefficients $c_j$ satisfy the condition \eqref{cond-c}.

In order to confirm our theoretical results, we take $\Delta t = h$ when $k=0$ and $\Delta t = h^{3/2}$ when $k=1$. The approximation errors of the Ensemble HDG method  are listed in \Cref{table_1} and the observed convergence rates match our theory. 	
\begin{table}
	\caption{History of convergence}\label{table_1}
	 [Errors for $\bm{q}_1$ and $u_1$]{
		\begin{tabular}{c|c|c|c|c|c}
			\Xhline{0.1pt}
			
			\multirow{2}{*}{Degree}
			&\multirow{2}{*}{$\frac{h}{\sqrt{2}}$}	
			&\multicolumn{2}{c|}{$E {\bm q}_1$}	
			&\multicolumn{2}{c}{$E u_1$}		\\
			\cline{3-6}
			& &Error &Rate
			&Error &Rate
			\\
			\cline{1-6}
			\multirow{5}{*}{ $k=0$}
			&$2^{-1}$	&1.8083E-02	    &	              &7.5266E-03	    &	      \\
			&$2^{-2}$	&1.0533E-02	    &0.77972	&1.7867E-03      &2.0747 	 \\
			&$2^{-3}$	&5.4869E-03	     &0.94087 	& 4.6561E-04	 &1.9401 	   \\
			&$2^{-4}$	&2.8018E-03	     &0.96967	& 1.1836E-04 	 &1.9760 	   \\
			&$2^{-5}$	&1.4106E-03     &0.98999	& 2.9742E-05    &1.9926 	   \\
			
			\cline{1-6}
			\multirow{5}{*}{ $k=1$}
			&$2^{-1}$	&6.6885E-03    &	       & 8.9039E-04	    &	      \\
			&$2^{-2}$	&6.8683E-04	    &3.2837	& 9.1693E-05      &3.2796 	 \\
			&$2^{-3}$	&1.1837E-04	    &2.5366 	& 1.1605E-05	 &2.9820 	   \\
			&$2^{-4}$	&2.5556E-05 	&2.2116	&  1.4488E-06	 & 3.0019 	   \\
			&$2^{-5}$	&6.2136E-06     &2.0402	&  1.8153E-07    &2.9966 	   \\

			\Xhline{0.1 pt}

		\end{tabular}
	}

	 [Errors for $\bm{q}_1$ and $u_1$]{
		\begin{tabular}{c|c|c|c|c|c}
			\Xhline{0.1pt}

			\multirow{2}{*}{Degree}
			&\multirow{2}{*}{$\frac{h}{\sqrt{2}}$}	
			&\multicolumn{2}{c|}{$E {\bm q}_1$}	
			&\multicolumn{2}{c}{$E u_1$}		\\
			\cline{3-6}
			& &Error &Rate
			&Error &Rate
			\\
			\cline{1-6}
			\multirow{5}{*}{ $k=0$}
			&$2^{-1}$	&3.7010E-02	    &	              & 1.5621E-02	    &	      \\
			&$2^{-2}$	& 2.2823E-02    &0.69743	& 4.7235E-03      &1.7256  	 \\
			&$2^{-3}$	&1.2037E-02	   &0.92305 	& 1.1678E-03 	 &2.0161 	   \\
			&$2^{-4}$	& 6.1239E-03	&0.97490 	&  2.9335E-04	 &1.9931	   \\
			&$2^{-5}$	&3.0784E-03     &0.99225	&  7.3533E-05   &1.9961 	   \\
			
			\cline{1-6}
			\multirow{5}{*}{ $k=1$}
			&$2^{-1}$	&1.5396E-03	    &	       &6.1243E-04	    &	      \\
			&$2^{-2}$	&6.9243E-04 	    &1.1528	&1.7340E-04      &1.8204	 \\
			&$2^{-3}$	&1.1603E-04   	& 2.5771	&2.3281E-05	 &2.8969 	   \\
			&$2^{-4}$	& 2.2315E-05 	&2.3784	& 2.7999E-06 	 & 3.0557 	   \\
			&$2^{-5}$	& 5.1675E-06    & 2.1105	&  3.5064E-07    & 2.9973 	   \\

			\Xhline{0.1 pt}

		\end{tabular}
	}

	 [Errors for $\bm{q}_1$ and $u_1$]{
		\begin{tabular}{c|c|c|c|c|c}
			\Xhline{0.1pt}

			\multirow{2}{*}{Degree}
			&\multirow{2}{*}{$\frac{h}{\sqrt{2}}$}	
			&\multicolumn{2}{c|}{$E {\bm q}_1$}	
			&\multicolumn{2}{c}{$E u_1$}		\\
			\cline{3-6}
			& &Error &Rate
			&Error &Rate
			\\
			\cline{1-6}
			\multirow{5}{*}{ $k=0$}
			&$2^{-1}$	& 5.5804E-02	&	                &2.1027E-02 	 &	      \\
			&$2^{-2}$	&3.0016E-02    &0.89461	     &6.4094E-03      &1.7140	 \\
			&$2^{-3}$	&1.5676E-02 	&0.93717 	&  1.5422E-03	 &2.0552 	   \\
			&$2^{-4}$	&7.9953E-03	    &0.97136 	&  3.8658E-04	 &1.9962 	   \\
			&$2^{-5}$	&4.0241E-03     &0.99049	& 9.6789E-05    &1.9978 	   \\
			
			\cline{1-6}
			\multirow{5}{*}{ $k=1$}
			&$2^{-1}$	&3.1068E-02	    &	       &2.4217E-03	    &	      \\
			&$2^{-2}$	&1.2610E-03 	    &4.6228	& 2.1045E-04     &3.5245 	 \\
			&$2^{-3}$	&2.0956E-04 	&2.5891 	& 2.6310E-05 	 &2.9998 	   \\
			&$2^{-4}$	&4.4629E-05 	&2.2313 	&  3.1994E-06	 &3.0397 	   \\
			&$2^{-5}$	&1.0818E-05     & 2.0446	& 4.0049E-07   &2.9980 	   \\

			\Xhline{0.1 pt}

		\end{tabular}
	}

\end{table}

\section{Conclusion}
In this work, we devised a new superconvergent Ensemble HDG method for parameterized convection diffusion PDEs.  This new Ensemble HDG method shares one common coefficient matrix and multiple RHS vectors, which is more efficient than performing separate simulations. We obtained a $L^{\infty}(0,T;L^2(\Omega))$ superconvergent rate for the solutions for all polynomial degree $k\ge 0$. As far as we are aware, this is the first time in the literature, it is even the first time for a single convection diffusion PDE to obtain the superconvergence rate when $k=0$.

\section*{Funding}
National Natural Science Foundation of China (NSFC) [grant no. 11801063 to G.\ Chen], China Postdoctoral Science Foundation [grant no. 2018M633339 and grant no. 2019T120808 to G.\ Chen], and US National Science Foundation (NSF) [grant no. DMS-1619904 to Y.\ Zhang].

\section{Appendix}

\label{app}
In this section,  we only give a proof of $\|\bm q_j- \overline{\bm q}_{jh}\|_{\mathcal{T}_h} \le  C h^{k+1}$, $\| u_j- \overline{ u }_{jh}\|_{\mathcal{T}_h} \le  C h^{k+2}$ and $	\|h_K^{1/2}( \overline{u}_{jh}-\widehat{\overline u}_{jh})\|_{\partial\mathcal{T}_h} \le Ch^{k+1}$ since the rest are similar. To prove the rest, we differentiate the error equations in \Cref{projection} with respect to time $t$. It is worth mentioning that we don't  need to assume that the coefficients are independent of time. However, we need to assume the coefficients are independent of time in the previous work \cite{PiChenZhangXu1}.

To shorten lengthy equations, we define the following HDG operators $ \mathscr B_j $ and $ \mathscr C_j $.
\begin{equation}\label{def_B}
\begin{split}
\hspace{0.1em}&\hspace{-0.1em}\mathscr  B_j(\overline{\bm q}_{jh}, \overline u_{jh},\widehat{ \overline{u}}_{jh};\bm r_j,w_j,\mu_j)\\
&=(c_j\overline{\bm{q}}_{jh}, \bm{r}_j)_{{\mathcal{T}_h}}- (\overline u_{jh}, \nabla\cdot \bm{r}_j)_{{\mathcal{T}_h}}+\langle \widehat{\overline u}_{jh}, \bm{r}_j\cdot \bm{n} \rangle_{\partial{{\mathcal{T}_h}}}\\
&\quad+ (\nabla\cdot\overline{\bm{q}}_{jh},  w_j)_{{\mathcal{T}_h}}
- \langle \overline{\bm q}_{jh}\cdot\bm n, \mu_j \rangle_{\partial \mathcal T_h}
+\langle  h_K^{-1} (P_M \overline{u}_{jh} - \widehat {\overline u}_{jh}) , P_M w_j-\mu_j \rangle_{\partial{{\mathcal{T}_h}}} \\
&\quad  -( \bm \beta_j  \overline u_{jh}, \nabla w_j)_{{\mathcal{T}_h}} - (\nabla\cdot\bm \beta_j  \overline u_{jh}, w_j)_{\mathcal T_h}+\langle \bm{\beta}_j\cdot\bm{n} \widehat {\overline u}_{jh}, w_j \rangle_{\partial{{\mathcal{T}_h}}}\\
\hspace{1em}&\hspace{-1em}\mathscr  C_j(\overline{\bm q}_{jh}, \overline u_{jh},\widehat{ \overline{u}}_{jh};\bm r_j,w_j,\mu_j)\\
&=(c_j\overline{\bm{q}}_{jh}, \bm{r}_j)_{{\mathcal{T}_h}}- (\overline u_{jh}, \nabla\cdot \bm{r}_j)_{{\mathcal{T}_h}}+\langle \widehat{\overline u}_{jh}, \bm{r}_j\cdot \bm{n} \rangle_{\partial{{\mathcal{T}_h}}}\\
&\quad+ (\nabla\cdot\overline{\bm{q}}_{jh},  w_j)_{{\mathcal{T}_h}}
- \langle \overline{\bm q}_{jh}\cdot\bm n, \mu_j \rangle_{\partial \mathcal T_h}
+\langle  h_K^{-1} (P_M \overline{u}_{jh} - \widehat {\overline u}_{jh}) , P_M w_j-\mu_j \rangle_{\partial{{\mathcal{T}_h}}} \\
&\quad  +( \bm \beta_j  \overline u_{jh}, \nabla w_j)_{{\mathcal{T}_h}} -\langle \bm{\beta}_j\cdot\bm{n} \widehat {\overline u}_{jh}, w_j \rangle_{\partial{{\mathcal{T}_h}}}.
\end{split}
\end{equation}
By the definition of \eqref{def_B},  we can rewrite the HDG formulation of the system \eqref{projection}, as follows: find $(\overline{\bm{q}}_{jh},\overline u_{jh},\widehat {\overline u}_{jh})\in \bm{V}_h\times W_h\times M_h(g_j)$  such that
\begin{align}\label{HDG_full_discrete}
	\mathscr B_j (\overline{\bm{q}}_{jh},\overline u_{jh},\widehat {\overline u}_{jh};\bm r_j,w_j,\mu_j) = (f_j-\partial_tu_j,w_j)_{\mathcal T_h},
\end{align}
for all $\left(\bm{r}_j,w_j,\mu_j\right)\in \bm V_h\times W_h\times M_h(0)$.

In the next lemmas, we present  some basic properties  of the operators $ \mathscr B_j $ and $\mathscr C_j$.
\begin{lemma}\label{energy_norm}
	For any $ ( \overline{\bm v}_{jh}, \overline w_{jh},\overline{\mu}_{jh} ) \in \bm V_h \times W_h \times M_h(0)$, we have
	\begin{align*}
		\hspace{1em}&\hspace{-1em} \mathscr B_j (\overline{\bm v}_{jh},  \overline w_{jh}, \overline{\mu}_{jh}; \overline{\bm v}_{jh}, \overline w_{jh},\overline{\mu}_{jh})\\
		& = (c_j\overline{\bm{v}}_{jh}, \overline{\bm{v}}_{jh})_{{\mathcal{T}_h}}+\langle h^{-1}_K (P_M \overline{w}_{jh}- \overline{\mu}_{jh}), P_M\overline{w}_{jh}- \overline{\mu}_{jh}\rangle_{\partial{{\mathcal{T}_h}}} \\
		&\quad - \frac{1}{2} \langle \bm\beta_j\cdot\bm n(\overline{w}_{jh}- \overline{\mu}_{jh}), \overline{w}_{jh}- \overline{\mu}_{jh} \rangle_{\partial\mathcal T_h} - \frac 1 2 (\nabla \cdot\bm \beta_j \overline{w}_{jh}, \overline{w}_{jh})_{\mathcal T_h}.
	\end{align*}
\end{lemma}

\begin{lemma}\label{eq_B}
	For any $ (\overline{\bm v}_{jh},\overline{w}_{jh},\widehat {\overline u}_{jh};\overline {\bm p}_{jh},\overline z_{jh},\widehat {\overline z}_{jh})   \in \bm V_h \times W_h \times M_h(0) \times \bm V_h \times W_h \times M_h(0)$, we have
	\begin{align*}
		\hspace{1em}&\hspace{-1em}\mathscr B_j (\overline{\bm v}_{jh},\overline{w}_{jh},\widehat {\overline u}_{jh};\overline {\bm p}_{jh},-\overline z_{jh},-\widehat {\overline z}_{jh}) + \mathscr C_j (\overline {\bm p}_{jh},\overline z_{jh},\widehat {\overline z}_{jh};-\overline{\bm v}_{jh},\overline{w}_{jh},\widehat {\overline u}_{jh})\\
		&= \langle  \bm{\beta}_j\cdot\bm n (\overline{w}_{jh} - \widehat {\overline{w}}_{jh}), \overline z_{jh}-\widehat {\overline z}_{jh} \rangle_{\partial \mathcal T_h}.
	\end{align*}
\end{lemma}
\begin{proof}
	By definition:
	\begin{align*}
		\hspace{1em}&\hspace{-1em} \mathscr B_j (\overline{\bm v}_{jh},\overline{w}_{jh},\widehat {\overline u}_{jh};\overline {\bm p}_{jh},-\overline z_{jh},-\widehat {\overline z}_{jh}) + \mathscr C_j (\overline {\bm p}_{jh},\overline z_{jh},\widehat {\overline z}_{jh};-\overline{\bm v}_{jh},\overline{w}_{jh},\widehat {\overline u}_{jh})\\
		&=(c_j\overline{\bm v}_{jh}, \overline {\bm p}_{jh})_{{\mathcal{T}_h}}- (\overline{w}_{jh}, \nabla\cdot \overline {\bm p}_{jh})_{{\mathcal{T}_h}}+\langle \widehat {\overline{w}}_{jh}, \overline {\bm p}_{jh} \cdot \bm{n} \rangle_{\partial{{\mathcal{T}_h}}} - (\nabla\cdot \overline{\bm v}_{jh}, \overline z_{jh})_{{\mathcal{T}_h}} \\
		&\quad - \langle h_K^{-1}(P_M \overline{w}_{jh} - \widehat {\overline{w}}_{jh}) , P_M \overline z_{jh} -\widehat {\overline z}_{jh} \rangle_{\partial{{\mathcal{T}_h}}} + \langle  \overline{\bm v}_{jh} \cdot \bm n,\widehat {\overline z}_{jh}\rangle_{\partial \mathcal T_h} \\
		&\quad  +( \bm \beta_j \overline{w}_{jh}, \nabla \overline z_{jh})_{{\mathcal{T}_h}} + (\nabla\cdot\bm \beta_j \overline{w}_{jh}, \overline z_{jh})_{\mathcal T_h} - \langle \bm{\beta}_j\cdot\bm{n}\widehat {\overline{w}}_{jh},  \overline z_{jh} \rangle_{\partial{{\mathcal{T}_h}}}\\
		&\quad-(c_j\bm{p}_h, \overline{\bm v}_{jh})_{{\mathcal{T}_h}}+ (\overline z_{jh}, \nabla\cdot \overline{\bm v}_{jh})_{{\mathcal{T}_h}} -\langle \widehat{\overline z}_{jh}, \overline{\bm v}_{jh} \cdot \bm{n} \rangle_{\partial{{\mathcal{T}_h}}}  + (\nabla\cdot \overline{\bm{p}}_{jh},  \overline{w}_{jh})_{{\mathcal{T}_h}} \\
		&\quad+\langle h_K^{-1}(P_M \overline z_{jh} - \widehat {\overline z}_{jh}), P_M \overline{w}_{jh} -\widehat {\overline{w}}_{jh} \rangle_{\partial{{\mathcal{T}_h}}} + \langle\overline {\bm p}_{jh} \cdot\bm n, \widehat {\overline{w}}_{jh} \rangle_{\partial{{\mathcal{T}_h}}}\\
		&\quad + ( \bm \beta_j \overline z_{jh}, \nabla \overline{w}_{jh})_{{\mathcal{T}_h}} - \langle \bm{\beta}_j\cdot\bm{n} \widehat {\overline z}_{jh},  \overline{w}_{jh} \rangle_{\partial{{\mathcal{T}_h}}}\\
		& = ( \bm \beta_j \overline{w}_{jh}, \nabla \overline z_{jh})_{{\mathcal{T}_h}} + (\nabla\cdot\bm \beta_j \overline{w}_{jh}, \overline z_{jh})_{\mathcal T_h} - \langle \bm{\beta}_j\cdot\bm{n} \widehat {\overline{w}}_{jh},  \overline z_{jh} \rangle_{\partial{{\mathcal{T}_h}}}\\
		&\quad + ( \bm \beta_j \overline z_{jh}, \nabla \overline{w}_{jh})_{{\mathcal{T}_h}} - \langle \bm{\beta}_j\cdot\bm{n} \widehat {\overline z}_{jh},  \overline{w}_{jh} \rangle_{\partial{{\mathcal{T}_h}}}\\
		& = \langle  \bm{\beta}_j\cdot\bm n (\overline{w}_{jh} - \widehat {\overline{w}}_{jh}), \overline z_{jh}-\widehat {\overline z}_{jh} \rangle_{\partial \mathcal T_h}.
	\end{align*}
\end{proof}

\subsection{Proof of Main Result}
\subsubsection{Step 1:  Error equation} \label{subsec:proof_step1}
\begin{lemma}\label{lemma:step1_first_lemma}
	For $\varepsilon_{jh}^{\bm q} = \bm{\Pi}_k\bm q_j - \overline{\bm q}_{jh}$, $\varepsilon_{jh}^{u} = {\Pi}_{k+1} u_j - \overline u_{jh}$ and $\varepsilon_{jh}^{\widehat { u}} = P_M u_j - \widehat {\overline u}_{jh}$, we have
	\begin{align}\label{error_y}
		\begin{split}
			\hspace{1em}&\hspace{-1em}\mathscr B_j(\varepsilon^{\bm q}_{jh},\varepsilon^u_{jh},\varepsilon^{\widehat u}_{jh};\bm r_j,w_j,\mu_j)\\
			&=  \langle  (\bm \Pi_k \bm q_j - \bm q_j)\cdot\bm n, w_j-\mu_j \rangle_{\partial \mathcal T_h} +\left\langle h^{-1}_K ( \Pi_{k+1} u_j- u_j), P_M w_j - \mu_j\right\rangle_{\partial {\mathcal{T}_h}}\\
			&\quad -( \bm \beta (\Pi_{k+1} u_j - u_j), \nabla w_j)_{{\mathcal{T}_h}} - (\nabla\cdot\bm \beta (\Pi_{k+1} u_j - u_j), w_j)_{\mathcal T_h}\\
			&\quad +\langle \bm{\beta}\cdot\bm{n} (P_M u_j - u_j) , w_j-\mu_j \rangle_{\partial{{\mathcal{T}_h}}}.
		\end{split}
	\end{align}
\end{lemma}

\begin{proof}
	By the definition of operator $\mathscr B_j$  in \eqref{def_B}, we have
	\begin{align*}
		\hspace{0.1em}&\hspace{-0.1em}  \mathscr B_j (\bm\Pi_k \bm q_j,\Pi_{k+1} u_j,P_Mu_j,\bm r_j,w_j,\mu_j)\\
		&= (c_j\bm\Pi_k \bm q_j,\bm{r}_j)_{\mathcal{T}_h}-(\Pi_{k+1} u_j,\nabla\cdot\bm{r}_j)_{\mathcal{T}_h}+\left\langle P_M u_j, \bm{r}_j\cdot \bm{n}\right\rangle_{\partial\mathcal{T}_h}\\
		&\quad + (\nabla \cdot(\bm\Pi_k\bm q_j),  w_j)_{\mathcal{T}_h}+\left\langle h_K^{-1} ( \Pi_{k+1} u_j- u_j), P_M w_j - \mu_j\right\rangle_{\partial {\mathcal{T}_h}}  - \left\langle \bm\Pi_k \bm q_j\cdot \bm{n}, \mu_j\right\rangle_{\partial {\mathcal{T}_h}}\\
		&\quad-( \bm \beta_j \Pi_{k+1} u_j, \nabla w_j)_{{\mathcal{T}_h}} - (\nabla\cdot\bm \beta_j \Pi_{k+1} u_j, w_j)_{\mathcal T_h}+\langle \bm{\beta}_j\cdot\bm{n} P_M u_j , w_j \rangle_{\partial{{\mathcal{T}_h}}}\\
		&= (c_j (\bm \Pi_k \bm q_j - \bm q_j),\bm{r}_j)_{\mathcal{T}_h}  + (c_j\bm q_j,\bm{r}_j)_{\mathcal{T}_h}  -(u_j,\nabla\cdot\bm{r}_j)_{\mathcal{T}_h}+\left\langle u_j, \bm{r}_j\cdot \bm{n}\right\rangle_{\partial\mathcal{T}_h} \\
		&\quad  + \langle  \bm \Pi_k \bm q_j \cdot\bm n, w_j  \rangle_{\partial \mathcal T_h}- (\bm\Pi_k\bm q_j,  \nabla w_j)_{\mathcal{T}_h}   +\left\langle h^{-1}_K ( \Pi_{k+1} u_j- u_j), P_M w_j - \mu_j\right\rangle_{\partial {\mathcal{T}_h}} \\
		&\quad - \left\langle \bm\Pi_k \bm q_j \cdot \bm{n}, \mu_j\right\rangle_{\partial {\mathcal{T}_h}}-( \bm \beta_j (\Pi_{k+1} u_j - u_j), \nabla w_j)_{{\mathcal{T}_h}} + (\bm{\beta}_j \nabla u_j, w_j)_{\mathcal T_h}\\
		&\quad - (\nabla\cdot\bm \beta_j (\Pi_{k+1} u_j - u_j), w_j)_{\mathcal T_h}+\langle \bm{\beta}\cdot\bm{n} (P_M u_j - u_j) , w_j - \mu_j \rangle_{\partial{{\mathcal{T}_h}}}\\
		&=  (c_j (\bm \Pi_k \bm q_j - \bm q_j),\bm{r})_{\mathcal{T}_h}  +(\bm q_j,\bm{r}_j)_{\mathcal{T}_h}  -( u_j,\nabla\cdot\bm{r}_j)_{\mathcal{T}_h}+\left\langle u_j, \bm{r}_j\cdot \bm{n}\right\rangle_{\partial\mathcal{T}_h}\\
		&\quad + \langle  \bm q_j\cdot\bm n, w_j  \rangle_{\partial \mathcal T_h}  + \langle  (\bm \Pi_k \bm q_j - \bm q_j)\cdot\bm n, w_j  \rangle_{\partial \mathcal T_h} - (\bm q_j,  \nabla w_j)_{\mathcal{T}_h}   \\
		&\quad+\left\langle h^{-1}_K ( \Pi_{k+1} u_j- u_j), P_M w_j - \mu_j\right\rangle_{\partial {\mathcal{T}_h}}  - \left\langle (\bm\Pi_k \bm q_j - \bm q_j)\cdot \bm{n}, \mu_j\right\rangle_{\partial {\mathcal{T}_h}}\\
		&\quad - ( \bm \beta_j (\Pi_{k+1} u_j - u_j), \nabla w_j)_{{\mathcal{T}_h}} + (\bm{\beta}_j \nabla u_j, w_j)_{\mathcal T_h}\\
		&\quad- (\nabla\cdot\bm \beta_j (\Pi_{k+1} u_j - u_j), w_j)_{\mathcal T_h}+\langle \bm{\beta}_j\cdot\bm{n} (P_M u_j - u_j) , w_j-\mu_j \rangle_{\partial{{\mathcal{T}_h}}}.
	\end{align*}
	Note that the exact state $ u_j $ and exact flux $ \bm{q}_j $ satisfy
	\begin{align*}
		(c_j\bm{q}_j,\bm{r}_j)_{\mathcal{T}_h}-(u_j,\nabla\cdot \bm{r}_j)_{\mathcal{T}_h}+\left\langle{u}_j,\bm r_j\cdot \bm n \right\rangle_{\partial {\mathcal{T}_h}} &= 0,\\
		-(\bm{q}_j,\nabla w_j)_{\mathcal{T}_h}+\left\langle {\bm{q}}_j\cdot \bm{n},w_j\right\rangle_{\partial {\mathcal{T}_h}} + (\bm{\beta}_j \nabla u_j, w_j)_{\mathcal T_h}&= (f_j - \partial_t u_j,w_j)_{\mathcal T_h},\\
		\left\langle {\bm{q}}_j\cdot \bm{n},\mu_j\right\rangle_{\partial {\mathcal{T}_h}} &=0
	\end{align*}
	for all $(\bm{r}_j,w_j,\mu_j)\in\bm{V}_h\times W_h\times M_h(0)$. Then we have
	\begin{align}\label{error_eq2}
		\begin{split}
			\hspace{1em}&\hspace{-1em} \mathscr B_j  (\bm\Pi_k \bm q_j,\Pi_{k+1} u_j ,P_Mu_j,\bm r_j,w_j,\mu_j) \\
			&=  (c_j (\bm \Pi_k \bm q_j - \bm q_j),\bm{r}_j)_{\mathcal{T}_h} + \langle  (\bm \Pi_k \bm q_j - \bm q_j)\cdot\bm n, w_j-\mu_j \rangle_{\partial \mathcal T_h} + (f_j - \partial_t u_j,w_j)_{\mathcal T_h} \\
			&\quad +\left\langle h^{-1}_K ( \Pi_{k+1} u_j- u_j), P_M w_j - \mu_j\right\rangle_{\partial {\mathcal{T}_h}}  -( \bm \beta_j (\Pi_{k+1} u_j - u_j), \nabla w_j)_{{\mathcal{T}_h}} \\
			&\quad - (\nabla\cdot\bm \beta_j (\Pi_{k+1} u_j - u_j), w_j)_{\mathcal T_h}+\langle \bm{\beta}\cdot\bm{n} (P_M u_j - u_j) , w_j - \mu_j \rangle_{\partial{{\mathcal{T}_h}}},
		\end{split}
	\end{align}
	subtract \eqref{HDG_full_discrete} from \eqref{error_eq2}, we have  
	\begin{align*}
		\hspace{1em}&\hspace{-1em}\mathscr B_j(\varepsilon^{\bm q}_{jh},\varepsilon^u_{jh},\varepsilon^{\widehat u}_{jh};\bm r_j,w_j,\mu_j)\\
		&=  (c_j (\bm \Pi_k \bm q_j - \bm q_j),\bm{r})_{\mathcal{T}_h} + \langle  (\bm \Pi_k \bm q_j - \bm q_j)\cdot\bm n, w_j-\mu_j \rangle_{\partial \mathcal T_h} \\
		&\quad +\left\langle h^{-1}_K ( \Pi_{k+1} u_j- u_j), P_M w_j - \mu_j\right\rangle_{\partial {\mathcal{T}_h}}  -( \bm \beta_j (\Pi_{k+1} u_j - u_j), \nabla w_j)_{{\mathcal{T}_h}} \\
		&\quad - (\nabla\cdot\bm \beta_j (\Pi_{k+1} u_j - u_j), w_j)_{\mathcal T_h}+\langle \bm{\beta}_j\cdot\bm{n} (P_M u_j - u_j) , w_j - \mu_j \rangle_{\partial{{\mathcal{T}_h}}}.
	\end{align*}
\end{proof}

\subsubsection{Step 2: Estimate for $\varepsilon_h^{ q}$} 
The proof of the following lemma is similar to a result established in \cite{Qiu_Shi_Convection_Diffusion_JSC_2016} and hence is omitted.
\begin{lemma}\label{nabla_ine}
	For all $j=1,2\cdots, J$ and $(\varepsilon^u_{jh},\varepsilon_{jh}^{\widehat u}) \in W_h\times M_h(0)$, we have
	\begin{align*}
		\hspace{1em}&\hspace{-1em} \|\nabla\varepsilon^u_{jh}\|_{\mathcal T_h}+  \|h_K^{-1/2} ({\varepsilon_{jh}^u -\varepsilon_{jh}^{\widehat u}})\|_{\partial \mathcal T_h}\\
		& \le C  \|\varepsilon^{\bm q}_{jh}\|_{\mathcal T_h}+C\|h_K^{-1/ 2}(P_M\varepsilon^u_{jh}-\varepsilon^{\widehat u}_{jh})\|_{\partial\mathcal T_h} + C\|\bm{\Pi}_k \bm q_j - \bm q_j\|_{\mathcal T_h}.
	\end{align*}
\end{lemma}

The next lemma is based on energy arguments.
\begin{lemma}\label{lemma:step2_main_lemma1}
	For $h$ small enough, we have
	\begin{align*}
		\hspace{1em}&\hspace{-1em}\|\varepsilon_{jh}^{\bm{q}}\|_{\mathcal{T}_h}^2+\|h_K^{-1/2}({P_M\varepsilon_{jh}^u-\varepsilon_{jh}^{\widehat{u}}})\|_{\partial \mathcal T_h}^2 \\
		&\le C  \|\bm\Pi_k \bm q_j - \bm q_j\|_{\mathcal T_h}^2+ C \|h_K^{1/2}(\bm\Pi_k \bm q_j - \bm q_j)\|_{\partial\mathcal T_h}^2+ C\|h_K^{-1/2}(\Pi_{k+1} u_j - u_j) \|_{\partial \mathcal T_h}^2  \\
		&\quad +C\|h_K^{1/2}(P_Mu_j - u_j)\|_{\partial \mathcal T_h}^2 + C\|\Pi_{k+1} u_j - u_j \|_{ \mathcal T_h}^2.
	\end{align*}
\end{lemma}
\begin{proof}
	First, the basic property of $ \mathscr B_j $ in \Cref{energy_norm} and use $\nabla\cdot\bm \beta_j\le 0$ to get
	\begin{align*}
		\mathscr B_j(\varepsilon_{jh}^{\bm q},\varepsilon_{jh}^{u}, \varepsilon_{jh}^{\widehat u}; \varepsilon_{jh}^{\bm q},\varepsilon_{jh}^{ u}, \varepsilon_{jh}^{\widehat u}) &\ge (c_j\varepsilon_{jh}^{\bm{q}},\varepsilon_{jh}^{\bm{q}})_{\mathcal{T}_h}+ \|h_K^{-1/2}( {P_M\varepsilon_{jh}^u-\varepsilon_{jh}^{\widehat{u}}})\|_{\partial \mathcal T_h}^2 \\
		&\quad - \frac 12 
		\langle{\bm \beta}_j \cdot\bm n (\varepsilon_{jh}^u-\varepsilon_{jh}^{\widehat u}),  \varepsilon_{jh}^u-\varepsilon_{jh}^{\widehat u} \rangle_{\partial\mathcal{T}_h}. 
	\end{align*}
	Then, taking $(\bm r_j, w_j,\mu_j) = ( \varepsilon_{jh}^{\bm q},\varepsilon_{jh}^u,\varepsilon_{jh}^{\widehat u})$ in \eqref{error_y} and the 
	stability \eqref{V1} with $(\gamma,w,\mu)=(\bm\beta_j,\varepsilon_{jh}^u,\varepsilon_{jh}^{\widehat u})$, we have 
	\begin{align*}
		\hspace{1em}&\hspace{-1em}(c_j\varepsilon_{jh}^{\bm{q}},\varepsilon_{jh}^{\bm{q}})_{\mathcal{T}_h}+\frac{1}{2} \|h_K^{-1/2}({P_M\varepsilon_{jh}^u-\varepsilon_{jh}^{\widehat{u}}})\|_{\partial \mathcal T_h}^2\\
		&\le Ch\|\nabla \varepsilon_{jh}^u\|^2_{\mathcal{T}_h} +  (c_j (\bm \Pi_k \bm q_j - \bm q_j), \varepsilon_{jh}^{\bm{q}})_{\mathcal{T}_h}  +  \langle (\bm \Pi_k \bm q_j - \bm q_j)\cdot \bm n,\varepsilon_{jh}^u - \varepsilon_{jh}^{\widehat u} \rangle_{\partial\mathcal T_h}\\
		&\quad -  \langle h_K^{-1} (\Pi_{k+1} u_j -u_j ),P_M \varepsilon_{jh}^u - \varepsilon_{jh}^{\widehat u}\rangle_{\partial\mathcal T_h}-( \bm \beta_j (\Pi_{k+1} u_j - u_j), \nabla \varepsilon_{jh}^u)_{{\mathcal{T}_h}}\\
		&\quad   - (\nabla\cdot\bm \beta_j (\Pi_{k+1} u_j - u_j), \varepsilon_{jh}^u)_{\mathcal T_h}+\langle \bm{\beta}_j\cdot\bm{n} (P_M u_j - u_j) , \varepsilon_{jh}^u  - \varepsilon_{jh}^{\widehat u}\rangle_{\partial{{\mathcal{T}_h}}}\\
		& =:\sum_{i=1}^7R_i.
	\end{align*}
	Next, we estimate $\{R_i\}_{i=1}^7$ term by term. First, by \Cref{nabla_ine} and Young's inequality, we have 
	\begin{align*}
		R_1 &\le Ch\|\varepsilon^{\bm q}_{jh}\|_{\mathcal T_h}^2 +Ch\|h_K^{-1/ 2}(P_M\varepsilon^u_{jh}-\varepsilon^{\widehat u}_{jh})\|_{\partial\mathcal T_h}^2 + C\|\bm{\Pi}_k \bm q_j - \bm q_j\|_{\mathcal T_h}^2,\\
		R_3 &\le C\|h_K^{1/2}(\bm\Pi_k \bm q_j - \bm q_j)\|_{\partial\mathcal T_h}^2 + \frac{1}{16}\|\varepsilon^{\bm q}_{jh}\|_{\mathcal T_h}^2\\
		&\quad  +\frac{1}{16}\|h_K^{-1/ 2}(P_M\varepsilon^u_{jh}-\varepsilon^{\widehat u}_{jh})\|_{\partial\mathcal T_h}^2+ C\|\bm{\Pi}_k \bm q_j - \bm q_j\|_{\mathcal T_h}^2,\\
		R_5 &\le C \|\Pi_{k+1} u_j - u_j \|_{\mathcal T_h}^2 + \frac{1}{16}\|\varepsilon^{\bm q}_{jh}\|_{\mathcal T_h}^2 \\
		&\quad +\frac{1}{16} \|h_K^{-1/ 2}(P_M\varepsilon^u_{jh}-\varepsilon^{\widehat u}_{jh})\|_{\partial\mathcal T_h}^2+ C\|\bm{\Pi}_k \bm q_j - \bm q_j\|_{\mathcal T_h}^2,\\
		R_7 &\le C\|h_K^{1/2}(P_Mu_j - u_j)\|_{\partial \mathcal T_h} \|h_K^{-1/2} (\varepsilon_{jh}^u - \varepsilon_{jh}^{\widehat u})\|_{\partial{{\mathcal{T}_h}}}\\
		&\le C\|h_K^{1/2}(P_Mu_j - u_j)\|_{\partial \mathcal T_h}^2 +  \frac{1}{16}\|\varepsilon^{\bm q}_{jh}\|_{\mathcal T_h}^2\\
		&\quad  +\frac{1}{16}\|h_K^{-1/ 2}(P_M\varepsilon^u_{jh}-\varepsilon^{\widehat u}_{jh})\|_{\partial\mathcal T_h}^2+ C\|\bm{\Pi}_k \bm q_j - \bm q_j\|_{\mathcal T_h}^2.
	\end{align*}
	Young's inequality for the terms $R_2$ and $R_4$,
	\begin{align*}
		R_2 &\le C\|\bm\Pi_k \bm q_j - \bm q_j\|_{\mathcal T_h}^2 + \frac{1}{16}\|\varepsilon^{\bm q}_{jh}\|_{\mathcal T_h}^2\\
		R_4 &\le C \|h_K^{-1/2}(\Pi_{k+1} u_j - u_j) \|_{\partial \mathcal T_h}^2  +\frac{1}{16} \|h_K^{-1/ 2}(P_M\varepsilon^u_{jh}-\varepsilon^{\widehat u}_{jh})\|_{\partial\mathcal T_h}^2.
	\end{align*}
	For the term $R_6$, using the Poincar\'e inequality \Cref{poincare} and \Cref{nabla_ine}, we have 
	\begin{align*}
		R_6 &\le C\|\Pi_{k+1} u_j - u_j \|_{ \mathcal T_h}(\|\nabla \varepsilon_{jh}^u\|_{\mathcal T_h} +  \|h_K^{-1/2}(\varepsilon^u_{jh}-\varepsilon^{\widehat u}_{jh})\|_{\partial\mathcal T_h})\\
		&\le C \|\Pi_{k+1} u_j - u_j \|_{ \mathcal T_h}^2   + \frac{1}{16}\|\varepsilon^{\bm q}_{jh}\|_{\mathcal T_h}^2\\
		&\quad  +\frac{1}{16}\|h_K^{-1/ 2}(P_M\varepsilon^u_{jh}-\varepsilon^{\widehat u}_{jh})\|_{\partial\mathcal T_h}^2+ C\|\bm{\Pi}_k \bm q_j - \bm q_j\|_{\mathcal T_h}^2.
	\end{align*}	
	Sum all the estimates above
	, and let $h$ small enough, we get 
	\begin{align*}
		\hspace{1em}&\hspace{-1em}\|\varepsilon_{jh}^{\bm{q}}\|_{\mathcal{T}_h}^2+\|h_K^{-1/2}({P_M\varepsilon_{jh}^u-\varepsilon_{jh}^{\widehat{u}}})\|_{\partial \mathcal T_h}^2 \\
		&\le C \|\bm\Pi_k \bm q_j - \bm q_j\|_{\mathcal T_h}^2+ C\|h_K^{1/2}(\bm\Pi_k \bm q_j - \bm q_j)\|_{\partial\mathcal T_h}^2+ C\|h_K^{-1/2}(\Pi_{k+1} u_j - u_j) \|_{\partial \mathcal T_h}^2  \\
		&\quad + C\|h_K^{1/2}(P_Mu_j - u_j)\|_{\partial \mathcal T_h}^2 + C\|\Pi_{k+1} u_j - u_j \|_{ \mathcal T_h}^2.
	\end{align*}
\end{proof}

As a consequence, a simple application of the triangle inequality gives optimal convergence rates for $\|\bm q_j -\overline{\bm q}_{jh}\|_{\mathcal T_h}$:
\begin{lemma}\label{lemma:step2_conv_rates}
	We have
	\begin{align}
		\|\bm q_j -\overline{\bm q}_{jh}\|_{\mathcal T_h}&\le \| \bm q_j - \bm \Pi_k \bm q_j\|_{\mathcal T_h} + \|\bm \Pi_k \bm q_j - \overline{\bm q}_{jh}\|_{\mathcal T_h} \le C  h^{k+1}.
	\end{align}
\end{lemma}

\subsubsection{Step 3: Estimate for $\varepsilon_{jh}^{u}$ by a duality argument}  

The next step is the consideration of the dual problems:

\begin{align}
	c_j\bm{\Phi}_j+\nabla\Psi_j&=0&\text{in}\ \Omega,\nonumber\\
	\nabla\cdot\bm \Phi_j-\bm{\beta}_j\cdot\nabla\Psi_j  &=\Theta_j&\text{in}\ \Omega,\label{Dual_PDE1}\\
	\Psi_j &= 0&\text{on}\ \partial\Omega\nonumber.
\end{align}

{\bfseries{Elliptic regularity}.} Since the domain $\Omega$ is convex, we have the following regularity estimate
\begin{align}\label{reg_e}
	\norm{\bm\Phi_j}_{[H^1(\Omega)]^d} + \norm{\Psi_j}_{H^2(\Omega)} \le C_{\textup{reg}} \norm{\Theta_j}_{L^2(\Omega)}.
\end{align}

With the above dual problems \eqref{Dual_PDE1} and regularity \eqref{reg_e}, we can derive the following error estimates.
\begin{lemma}\label{lemma:step3_first_lemma}
	For $h$ small enough, we have
	\begin{align*}
		\|\varepsilon^u_{jh}\|_{\mathcal{T}_h}
		&\le C h^{3/2} \|\bm \Pi_k \bm q_j - \bm q_j\|_{\partial\mathcal  T_h} + Ch \|\bm \Pi_k \bm q_j - \bm q_j\|_{\mathcal  T_h}+Ch  \|\varepsilon_{jh}^{\bm q}\|_{\mathcal T_h}\\
		&\quad + C h\| h_K^{-1}(\Pi_{k+1} u_j -  u_j)\|_{\partial\mathcal  T_h}+ Ch\|h_K^{-1/2}(\varepsilon_{jh}^u - \varepsilon_{jh}^{\widehat u})\|_{\partial\mathcal  T_h}   \\
		&\quad  + Ch\|h_K^{-1/2}(P_M\varepsilon_{jh}^u - \varepsilon_{jh}^{\widehat u})\|_{\partial\mathcal  T_h} + Ch^{3/2}\|P_M u_j - u_j \|_{\partial \mathcal T_h} \\
		&\quad + C\|\Pi_{k+1} u_j -u_j\|_{\mathcal T_h}.
	\end{align*}
\end{lemma}

\begin{proof}
	Consider the dual problem \eqref{Dual_PDE1} and let $\Theta_j = \varepsilon_{jh}^u$, we take  $(\bm r_j,w_j,\mu_j) = (-\bm\Pi_k\bm{\Phi}_j,\Pi_{k+1}\Psi_j,P_M\Psi_j)$ in \Cref{error_y} in  \Cref{lemma:step1_first_lemma},  we have
	\begin{align}\label{one}
		\begin{split}
			\hspace{0.1em}&\hspace{-0.1em}\mathscr B_j (\varepsilon^{\bm q}_{jh},\varepsilon^u_{jh},\varepsilon^{\widehat u}_{jh};-\bm\Pi_k\bm{\Phi}_j,\Pi_{k+1}\Psi_j,P_M\Psi_j)\\
			&=\mathscr C_j(\bm\Pi_k\bm{\Phi}_j,\Pi_{k+1}\Psi_j,P_M\Psi_j;- \varepsilon^{\bm q}_{jh},\varepsilon^u_{jh},\varepsilon^{\widehat u}_{jh}) \\
			&\quad + \langle \bm{\beta}_j\cdot\bm n(\varepsilon_{jh}^u - \varepsilon_{jh}^{\widehat u}), \Pi_{k+1}\Psi_j - P_M\Psi_j   \rangle_{\partial\mathcal T_h}\\
			&= -(c_j (\bm \Pi_k \bm \Phi_j - \bm \Phi_j),\varepsilon^{\bm q}_{jh})_{\mathcal{T}_h} +\langle  (\bm \Pi_k \bm \Phi_j - \bm \Phi_j)\cdot\bm n, \varepsilon_{jh}^u-  \varepsilon_{jh}^{\widehat u} \rangle_{\partial \mathcal T_h} \\
			&\quad +\left\langle h^{-1}_K ( \Pi_{k+1} \Psi_j- \Psi_j), P_M \varepsilon_{jh}^u - \varepsilon_{jh}^{\widehat u}\right\rangle_{\partial {\mathcal{T}_h}} + \|\varepsilon_{jh}^u\|_{\mathcal T_h}^2\\
			&\quad + ( \bm \beta_j (\Pi_{k+1} \Psi_j- \Psi_j), \nabla \varepsilon_{jh}^u)_{{\mathcal{T}_h}} -\langle \bm{\beta}_j\cdot\bm{n} (P_M \Psi_j - \Psi_j), \varepsilon_{jh}^u  - \varepsilon_{jh}^{\widehat u} \rangle_{\partial{{\mathcal{T}_h}}} \\
			&\quad + \langle \bm{\beta}_j\cdot\bm n(\varepsilon_{jh}^u - \varepsilon_{jh}^{\widehat u}), \Pi_{k+1}\Psi_j - P_M\Psi_j   \rangle_{\partial\mathcal T_h}. 
		\end{split}
	\end{align}
	On the other hand, by \eqref{error_y}, we have
	\begin{align*}
		\hspace{1em}&\hspace{-1em}\mathscr B_j(\varepsilon^{\bm q}_{jh},\varepsilon^u_{jh},\varepsilon^{\widehat u}_{jh};-\bm\Pi_k\bm{\Phi}_j,\Pi_{k+1}\Psi_j,P_M\Psi_j) \\
		&= - (c_j (\bm \Pi_k \bm q_j - \bm q_j),\bm\Pi_k\bm{\Phi}_j)_{\mathcal{T}_h} + \langle  (\bm \Pi_k \bm q_j - \bm q_j)\cdot\bm n, \Pi_{k+1} \Psi_j - P_M \Psi_j\rangle_{\partial \mathcal T_h} \\
		&\quad +\left\langle h^{-1}_K ( \Pi_{k+1} u_j- u_j), P_M \Pi_{k+1} \Psi_j - P_M \Psi_j\right\rangle_{\partial {\mathcal{T}_h}}-( \bm \beta_j (\Pi_{k+1} u_j - u_j), \nabla \Pi_{k+1}\Psi_j)_{{\mathcal{T}_h}} \\
		&\quad- (\nabla\cdot\bm \beta_j (\Pi_{k+1} u_j - u_j), \Pi_{k+1}\Psi_j)_{\mathcal T_h}+\langle \bm{\beta}_j\cdot\bm{n} (P_M u_j - u_j) , \Pi_{k+1}\Psi_j  - P_M\Psi_j\rangle_{\partial{{\mathcal{T}_h}}}.
	\end{align*}
	Since there holds
	\begin{align*}
		&\langle
		(\bm{\Pi}_k^{o}\bm q_j-\bm q_j)\cdot\bm n,P_M\Psi
		\rangle_{\partial\mathcal{T}_h}
		\nonumber\\
		&\qquad=
		\langle\bm{\Pi}_k^{o}\bm q_j\cdot\bm n,P_M\Psi
		\rangle_{\partial\mathcal{T}_h}
		-\langle	\bm q_j\cdot\bm n,P_M\Psi
		\rangle_{\partial\mathcal{T}_h}
		=	\langle\bm{\Pi}_k\bm q_j\cdot\bm n,P_M\Psi
		\rangle_{\partial\mathcal{T}_h},\\
		&\langle
		\bm\beta_j\cdot\bm n(P_Mu_j-u_j),P_M\Psi_j
		\rangle_{\partial\mathcal{T}_h}=0=
		\langle
		\bm\beta_j\cdot\bm n(P_Mu_j-u_j),\Psi_j
		\rangle_{\partial\mathcal{T}_h}.
	\end{align*}
	This gives 
	\begin{align}\label{two}
		\begin{split}
			\hspace{0.1em}&\hspace{-0.1em}\mathscr B_j(\varepsilon^{\bm q}_{jh},\varepsilon^u_{jh},\varepsilon^{\widehat u}_{jh};-\bm\Pi_k\bm{\Phi}_j,\Pi_{k+1}\Psi_j,P_M\Psi_j) \\
			&= - (c_j (\bm \Pi_k \bm q_j - \bm q_j),\bm\Pi_k\bm{\Phi}_j)_{\mathcal{T}_h} + \langle  (\bm \Pi_k \bm q_j - \bm q_j)\cdot\bm n, \Pi_{k+1} \Psi_j -  \Psi_j\rangle_{\partial \mathcal T_h} \\
			&\quad +\left\langle h^{-1}_K ( \Pi_{k+1} u_j- u_j), P_M \Pi_{k+1} \Psi_j - P_M \Psi_j\right\rangle_{\partial {\mathcal{T}_h}}-( \bm \beta_j (\Pi_{k+1} u_j - u_j), \nabla \Pi_{k+1}\Psi_j)_{{\mathcal{T}_h}} \\
			&\quad- (\nabla\cdot\bm \beta_j (\Pi_{k+1} u_j - u_j), \Pi_{k+1}\Psi_j)_{\mathcal T_h}+\langle \bm{\beta}_j\cdot\bm{n} (P_M u_j - u_j) , \Pi_{k+1}\Psi_j  - \Psi_j\rangle_{\partial{{\mathcal{T}_h}}}.
		\end{split}
	\end{align}

	Comparing the above two equalities \eqref{one} and \eqref{two}, we have
	\begin{align*}
		\|\varepsilon^u_{jh}\|^2_{\mathcal{T}_h}
		&=  - (c_j (\bm \Pi_k \bm q_j - \bm q_j),\bm\Pi_k\bm{\Phi}_j)_{\mathcal{T}_h} + (c_j (\bm \Pi_k \bm \Phi_j - \bm \Phi_j),\varepsilon^{\bm q}_{jh})_{\mathcal{T}_h} \\
		&\quad +  \langle  (\bm \Pi_k \bm q_j - \bm q_j)\cdot\bm n, \Pi_{k+1} \Psi_j -  \Psi_j\rangle_{\partial \mathcal T_h} \\
		&\quad +\left\langle h^{-1}_K ( \Pi_{k+1} u_j- u_j), P_M \Pi_{k+1} \Psi_j - P_M \Psi_j\right\rangle_{\partial {\mathcal{T}_h}}\\
		&\quad -\langle  (\bm \Pi_k \bm \Phi_j - \bm \Phi_j)\cdot\bm n, \varepsilon_{jh}^u-  \varepsilon_{jh}^{\widehat u} \rangle_{\partial \mathcal T_h} \\
		&\quad -\left\langle h^{-1}_K ( \Pi_{k+1} \Psi_j- \Psi_j), P_M \varepsilon_{jh}^u - \varepsilon_{jh}^{\widehat u}\right\rangle_{\partial {\mathcal{T}_h}} \\
		&\quad  - ( \bm \beta_j (\Pi_{k+1} \Psi_j- \Psi_j), \nabla \varepsilon_{jh}^u)_{{\mathcal{T}_h}}+\langle \bm{\beta}_j \cdot\bm{n} (P_M \Psi_j - \Psi_j), \varepsilon_{jh}^u -  \varepsilon_{jh}^{\widehat u} \rangle_{\partial{{\mathcal{T}_h}}}\\
		&\quad  - \langle \bm{\beta}_j\cdot\bm n(\varepsilon_{jh}^u - \varepsilon_{jh}^{\widehat u}), \Pi_{k+1}\Psi_j - P_M\Psi_j   \rangle_{\partial\mathcal T_h}-( \bm \beta_j (\Pi_{k+1} u_j - u_j), \nabla \Pi_{k+1}\Psi_j)_{{\mathcal{T}_h}}\\
		&\quad - (\nabla\cdot\bm \beta_j (\Pi_{k+1} u_j - u_j), \Pi_{k+1}\Psi_j)_{\mathcal T_h}+\langle \bm{\beta}_j\cdot\bm{n} (P_M u_j - u_j) , \Pi_{k+1}\Psi_j-\Psi_j \rangle_{\partial{{\mathcal{T}_h}}}\\
		&=:\sum_{i=1}^{12} R_i.
	\end{align*}
	Next, we estimate $\{R_i\}_{i=1}^{12}$ term by term. First,
	\begin{align*}
		R_1 + R_2 &=  - ((c_j - \Pi_0 c_j) (\bm \Pi_k \bm q_j - \bm q_j),\bm\Pi_k\bm{\Phi}_j)_{\mathcal{T}_h} + (c_j (\bm \Pi_k \bm \Phi_j - \bm \Phi_j),\varepsilon^{\bm q}_{jh})_{\mathcal{T}_h} \\
		&\le Ch|c_j|_{1,\infty}\|\bm \Pi_k \bm q_j - \bm q_j\|_{\mathcal{T}_h}\|\varepsilon^u_{jh}\|_{\mathcal T_h} +Ch  \|\varepsilon_{jh}^{\bm q}\|_{\mathcal T_h} \|\varepsilon^u_{jh}\|_{\mathcal T_h}. 
	\end{align*}
	Then, we have 
	\begin{align*}
		\hspace{1em}&\hspace{-1em}R_3+ R_4 + R_5 + R_6 +R_9\\
		&\le C  h^{3/2} \left( \|\bm \Pi_k \bm q_j - \bm q_j\|_{\partial\mathcal  T_h} + \| h_K^{-1}(\Pi_{k+1} u_j -  u_j)\|_{\partial\mathcal  T_h}+ \|h_K^{-1/2}(\varepsilon_{jh}^u - \varepsilon_{jh}^{\widehat u})\|_{\partial\mathcal  T_h}  \right.  \\
		&\qquad \qquad\left.+C \|h_K^{-1/2}(P_M\varepsilon_{jh}^u - \varepsilon_{jh}^{\widehat u})\|_{\partial\mathcal  T_h}+ C\|h_K^{-1/2}(\varepsilon_{jh}^u - \varepsilon_{jh}^{\widehat u})\|_{\partial\mathcal  T_h} \right) \|\varepsilon^u_{jh}\|_{\mathcal{T}_h}.
	\end{align*}
	For the term $R_7$, by \Cref{poincare}, we get
	\begin{align*}
		R_7 \le C h^2 (\|\varepsilon^{\bm q}_{jh}\|_{\mathcal T_h}+\|h_K^{-1/ 2}(P_M\varepsilon^u_{jh}-\varepsilon^{\widehat u}_{jh})\|_{\partial\mathcal T_h})\|\varepsilon^u_{jh}\|_{\mathcal{T}_h}.
	\end{align*}
	For the terms $R_8$ and $R_{12}$, we have
	\begin{align*}
		R_8 + R_{12}  &= \langle \bm{\beta}_j\cdot\bm{n} (P_M \Psi_j - \Psi_j), \varepsilon_{jh}^u  \rangle_{\partial{{\mathcal{T}_h}}}+\langle \bm{\beta}\cdot\bm{n} (P_M u_j - u_j) , \Pi_{k+1}\Psi_j \rangle_{\partial{{\mathcal{T}_h}}}\\
		&= \langle \bm{\beta}_j\cdot\bm{n} (P_M \Psi_j - \Psi_j), \varepsilon_{jh}^u- \varepsilon_{jh}^{\widehat u} \rangle_{\partial{{\mathcal{T}_h}}}\\
		&\quad +\langle \bm{\beta}_j\cdot\bm{n} (P_M u_j - u_j) , \Pi_{k+1}\Psi_j - \Psi_j \rangle_{\partial{{\mathcal{T}_h}}}\\
		&\le C (h \|h_K^{-1/2} (P_M\varepsilon_{jh}^u - \varepsilon_{jh}^{\widehat u})\|_{\partial\mathcal T_h} + h^{3/2}\|P_M u_j - u_j \|_{\partial \mathcal T_h})\|\varepsilon^u_{jh}\|_{\mathcal{T}_h}.
	\end{align*}
	For the terms $R_{10}$ and $R_{11}$, we use the boundness of $\Pi_{k+1}$ to get
	\begin{align*}
		R_{10} + R_{11} &\le C \|\Pi_{k+1} u_j -u_j\|_{\mathcal T_h} (\|\nabla\Pi_{k+1} \Psi_j\|_{\mathcal T_h} + \|\Pi_{k+1} \Psi_j\|_{\mathcal T_h})\\
		&\le C\|\Pi_{k+1} u_j -u_j\|_{\mathcal T_h} (\|\nabla(\Pi_{k+1} \Psi_j - \Psi_j)\|_{\mathcal T_h}  +\|\nabla\Psi_j\|_{\mathcal T_h}+ \|\Pi_{k+1} \Psi_j\|_{\mathcal T_h})\\
		&\le C\|\Pi_{k+1} u_j -u_j\|_{\mathcal T_h} \|\varepsilon^u_{jh}\|_{\mathcal{T}_h}.
	\end{align*}
	
	Thus,  combining all the estimates above give 
	\begin{align*}
		\|\varepsilon^u_{jh}\|_{\mathcal{T}_h}
		&\le C  h^{3/2} \|\bm \Pi_k \bm q_j - \bm q_j\|_{\partial\mathcal  T_h} + Ch \|\bm \Pi_k \bm q_j - \bm q_j\|_{\mathcal  T_h}+Ch  \|\varepsilon_{jh}^{\bm q}\|_{\mathcal T_h}\\
		&\quad +  Ch\| h_K^{-1}(\Pi_{k+1} u_j -  u_j)\|_{\partial\mathcal  T_h}+ Ch\|h_K^{-1/2}(\varepsilon_{jh}^u - \varepsilon_{jh}^{\widehat u})\|_{\partial\mathcal  T_h}   \\
		&\quad  + Ch\|h_K^{-1/2}(P_M\varepsilon_{jh}^u - \varepsilon_{jh}^{\widehat u})\|_{\partial\mathcal  T_h} + Ch^{3/2}\|P_M u_j - u_j \|_{\partial \mathcal T_h} \\
		&\quad + C\|\Pi_{k+1} u_j -u_j\|_{\mathcal T_h}.
	\end{align*}
\end{proof}

As a consequence, a simple application of the triangle inequality gives optimal convergence rates for $\|u_j -\overline u_{jh}\|_{\mathcal T_h}$  and $	\|h_K^{1/2}( \overline{u}_{jh}-\widehat{\overline u}_{jh})\|_{\partial\mathcal{T}_h} $:

\begin{lemma}\label{lemma:step3_conv_rates}
	For $h$ small enough, we have
	\begin{align*}
		\|u_j -\overline u_{jh}\|_{\mathcal T_h}\le  \|\Pi_{k+1} u_j -u_j\|_{\mathcal T_h} + \| \Pi_{k+1} u_j - \overline u_{jh}\|_{\mathcal T_h} \le C h^{k+2},
	\end{align*}
	and 
	\begin{align*}
		\|h_K^{1/2}( \overline{u}_{jh}-\widehat{\overline u}_{jh})\|_{\partial\mathcal{T}_h} 
		&\le \|h_K^{1/2}\varepsilon_{jh}^n\|_{\partial\mathcal{T}_h}
		+\|h_K^{1/2}( P_Mu_j-\widehat{\overline u}_{jh})\|_{\partial\mathcal{T}_h}
		\\
		&\quad +\|h_K^{1/2}( \Pi_{k+1}u_j-P_Mu_j)\|_{\partial\mathcal{T}_h}\\
		&\le Ch^{k+1}.
	\end{align*}
\end{lemma}

\bibliographystyle{plain}
\bibliography{/Users/ywzfg/Dropbox/Research/Bib/Model_Order_Reduction,/Users/ywzfg/Dropbox/Research/Bib/Ensemble,/Users/ywzfg/Dropbox/Research/Bib/HDG,/Users/ywzfg/Dropbox/Research/Bib/Interpolatory,/Users/ywzfg/Dropbox/Research/Bib/Mypapers,/Users/ywzfg/Dropbox/Research/Bib/Added}

\begin{thebibliography}{10}

\bibitem{ChenCockburnSinglerZhang1}
Gang Chen, Bernardo Cockburn, John Singler, and Yangwen Zhang.
\newblock Superconvergent {I}nterpolatory {HDG} methods for reaction diffusion
  equations. {P}art {I}: {HDG}-k methods, https://arxiv.org/abs/1905.12055.
\newblock 2019.

\bibitem{Chen_Monk_Peter1}
Gang Chen, Peter Monk, and Yangwen Zhang.
\newblock An {HDG} {M}ethod for the {T}ime-dependent {D}rift--{D}iffusion
  {M}odel of {S}emiconductor {D}evices.
\newblock {\em J. Sci. Comput.}, 80(1):420--443, 2019.

\bibitem{PiChenZhangXu1}
Gang Chen, Liangya Pi, Liwei Xu, and Yangwen Zhang.
\newblock A {S}uperconvergent {E}nsemble {HDG} {M}ethod for {P}arameterized
  {C}onvection {D}iffusion {E}quations.
\newblock 2019.
\newblock Submitted to SIAM Journal on Numerical Analysis,
  https://arxiv.org/abs/1903.04017.

\bibitem{Chen_Cockburn_Convection_Diffusion_MathComp_2014}
Yanlai Chen and Bernardo Cockburn.
\newblock Analysis of variable-degree {HDG} methods for convection-diffusion
  equations. {P}art {II}: {S}emimatching nonconforming meshes.
\newblock {\em Math. Comp.}, 83(285):87--111, 2014.

\bibitem{Cockburn_Zhixing_Hungria_Waves_JSC_2018}
Bernardo Cockburn, Zhixing Fu, Allan Hungria, Liangyue Ji, Manuel~A. S\'anchez,
  and Francisco-Javier Sayas.
\newblock Stormer-{N}umerov {HDG} {M}ethods for {A}coustic {W}aves.
\newblock {\em J. Sci. Comput.}, 75(2):597--624, 2018.

\bibitem{Cockburn_Gopalakrishnan_Lazarov_Unify_SINUM_2009}
Bernardo Cockburn, Jayadeep Gopalakrishnan, and Raytcho Lazarov.
\newblock Unified hybridization of discontinuous {G}alerkin, mixed, and
  continuous {G}alerkin methods for second order elliptic problems.
\newblock {\em SIAM J. Numer. Anal.}, 47(2):1319--1365, 2009.

\bibitem{Cockburn_Shi_Stokes_MathComp_2013}
Bernardo Cockburn and Ke~Shi.
\newblock Conditions for superconvergence of {HDG} methods for {S}tokes flow.
\newblock {\em Math. Comp.}, 82(282):651--671, 2013.

\bibitem{Fiordilino_Bousinesq_SINUM_2018}
J.~A. Fiordilino.
\newblock A {S}econd {O}rder {E}nsemble {T}imestepping {A}lgorithm for
  {N}atural {C}onvection.
\newblock {\em SIAM J. Numer. Anal.}, 56(2):816--837, 2018.

\bibitem{Gunzburger_Jiang_Wang_Flow_CMAM_2017}
M.~Gunzburger, N.~Jiang, and Z.~Wang.
\newblock A second-order time-stepping scheme for simulating ensembles of
  parameterized flow problems.
\newblock {\em Comput. Methods Appl. Math.}, 2017.

\bibitem{Gunzburger_Jiang_Schneier_NS_SINUM_2017}
Max Gunzburger, Nan Jiang, and Michael Schneier.
\newblock An ensemble-proper orthogonal decomposition method for the
  nonstationary {N}avier-{S}tokes equations.
\newblock {\em SIAM J. Numer. Anal.}, 55(1):286--304, 2017.

\bibitem{Gunzburger_Jiang_Schneier_NS_IJNAM_2018}
Max Gunzburger, Nan Jiang, and Michael Schneier.
\newblock A higher-order ensemble/proper orthogonal decomposition method for
  the nonstationary {N}avier-{S}tokes equations.
\newblock {\em Int. J. Numer. Anal. Model.}, 15(4-5):608--627, 2018.

\bibitem{Gunzburger_Jiang_Wang_Flow_IMAJNA_2018}
Max Gunzburger, Nan Jiang, and Zhu Wang.
\newblock {An efficient algorithm for simulating ensembles of parameterized
  flow problems}.
\newblock {\em IMA Journal of Numerical Analysis}, 05 2018.

\bibitem{HuShenSinglerZhangZheng_HDG_Dirichlet_control3}
Weiwei Hu, Jiguang Shen, John~R. Singler, Yangwen Zhang, and Xiaobo Zheng.
\newblock A {S}uperconvergent {HDG} {M}ethod for {D}istributed {C}ontrol of
  {C}onvection {D}iffusion {PDE}s.
\newblock {\em J. Sci. Comput.}, 76(3):1436--1457, 2018.

\bibitem{Jiang_Fluid_JSC_2015}
Nan Jiang.
\newblock A higher order ensemble simulation algorithm for fluid flows.
\newblock {\em J. Sci. Comput.}, 64(1):264--288, 2015.

\bibitem{Jiang_NS_NUPDE_2017}
Nan Jiang.
\newblock A second-order ensemble method based on a blended backward
  differentiation formula timestepping scheme for time-dependent
  {N}avier-{S}tokes equations.
\newblock {\em Numer. Methods Partial Differential Equations}, 33(1):34--61,
  2017.

\bibitem{Jiang_Layton_Flow_UQ_2014}
Nan Jiang and William Layton.
\newblock An algorithm for fast calculation of flow ensembles.
\newblock {\em Int. J. Uncertain. Quantif.}, 4(4):273--301, 2014.

\bibitem{Jiang_Layton_Fluid_NMPDE_2015}
Nan Jiang and William Layton.
\newblock Numerical analysis of two ensemble eddy viscosity numerical
  regularizations of fluid motion.
\newblock {\em Numer. Methods Partial Differential Equations}, 31(3):630--651,
  2015.

\bibitem{Jiang_Tran_Flow_CMAM_2015}
Nan Jiang and Hoang Tran.
\newblock Analysis of a stabilized {CNLF} method with fast slow wave splittings
  for flow problems.
\newblock {\em Comput. Methods Appl. Math.}, 15(3):307--330, 2015.

\bibitem{Lehrenfeld_PhD_thesis_2010}
C.~Lehrenfeld.
\newblock {H}ybrid {D}iscontinuous {G}alerkin methods for solving
  incompressible flow problems.
\newblock 2010.
\newblock PhD Thesis.

\bibitem{Luo_Wang_Heat_SINUM_2018}
Yan Luo and Zhu Wang.
\newblock An {E}nsemble {A}lgorithm for {N}umerical {S}olutions to
  {D}eterministic and {R}andom {P}arabolic {PDE}s.
\newblock {\em SIAM J. Numer. Anal.}, 56(2):859--876, 2018.

\bibitem{Oikawa_Poisson_JSC_2015}
Issei Oikawa.
\newblock A hybridized discontinuous {G}alerkin method with reduced
  stabilization.
\newblock {\em J. Sci. Comput.}, 65(1):327--340, 2015.

\bibitem{Qiu_Shen_Shi_elasticity_MathComp_2018}
Weifeng Qiu, Jiguang Shen, and Ke~Shi.
\newblock An {HDG} method for linear elasticity with strong symmetric stresses.
\newblock {\em Math. Comp.}, 87(309):69--93, 2018.

\bibitem{Qiu_Shi_Convection_Diffusion_JSC_2016}
Weifeng Qiu and Ke~Shi.
\newblock An {HDG} method for convection diffusion equation.
\newblock {\em J. Sci. Comput.}, 66(1):346--357, 2016.

\bibitem{Qiu_Shi_NS_IMAJNA_2016}
Weifeng Qiu and Ke~Shi.
\newblock A superconvergent {HDG} method for the incompressible
  {N}avier-{S}tokes equations on general polyhedral meshes.
\newblock {\em IMA J. Numer. Anal.}, 36(4):1943--1967, 2016.

\bibitem{Sanchez_Ciuca_Nguyen_Peraire_Cockburn_Hamiltonian_JCP_2017}
M.~A. S\'anchez, C.~Ciuca, N.~C. Nguyen, J.~Peraire, and B.~Cockburn.
\newblock Symplectic {H}amiltonian {HDG} methods for wave propagation
  phenomena.
\newblock {\em J. Comput. Phys.}, 350:951--973, 2017.

\bibitem{Vidal_Nguyen_Peraire_electromagnetic_JCP_2018}
F.~Vidal-Codina, N.~C. Nguyen, S.-H. Oh, and J.~Peraire.
\newblock A hybridizable discontinuous {G}alerkin method for computing nonlocal
  electromagnetic effects in three-dimensional metallic nanostructures.
\newblock {\em J. Comput. Phys.}, 355:548--565, 2018.

\end{thebibliography}

\end{document}